%% file: caop.tex
\documentclass[12pt]{article}

\usepackage{times}

\usepackage[a4paper,text={128mm,185mm}]{geometry}

\usepackage{titlesec}

\titleformat{\section}[hang]%
{\bfseries\large}{\thesection.}{1ex}{}%

\titleformat{\subsection}[hang]%
{\bfseries}{\thesubsection}{1ex}{}%

\usepackage{amsthm}

\theoremstyle{plain}
\newtheorem{theorem}{Theorem}[section]
\newtheorem{proposition}[theorem]{Proposition}
\newtheorem{lemma}[theorem]{Lemma}

\theoremstyle{definition}
\newtheorem{definition}[theorem]{Definition}

\theoremstyle{remark}
\newtheorem{remark}[theorem]{Remark}

\usepackage{amsmath,%amsfonts,
amssymb,%latexsym,amsthm,
mathrsfs%,stmaryrd,euler
}

\usepackage[all]{xy}
\UseComputerModernTips

\usepackage{eepic}
\newcommand{\dashlinestretch}{30}

\usepackage{hyperref}
\usepackage{url}

\usepackage{macros}

\begin{document}

% Title + authors will be added by Cahiers,
% we must leave 3.5cm blank for that purpose:
%
%\mbox{ }
%\vspace{35mm}
%
%%%HACK!!!
\title{\Large An Axiomatics and a Combinatorial Model of Creation/Annihilation
	Operators}
\author{\large Marcelo Fiore}
\date{}
\maketitle

% Abstract in French and in English, followed by Keywords and MSC:
%
\begin{minipage}{118mm}{\small
%%%HACK!!!
%{\bf R\'esum\'e.} ...\\
{\bf Abstract.} 
A categorical axiomatic theory of creation/annihilation operators on
bosonic %(or symmetric) 
Fock space is introduced and the combinatorial model that motivated it is
presented.  Commutation relations and coherent states are considered in
both frameworks.\\
%%%HACK!!!
%{\bf Keywords.} Your keywords come here.\\
%{\bf Mathematics Subject Classification (2010).} Your MSC numbers come here.
}\end{minipage}

% Here starts the text:

\section*{Introduction}

This work is an investigation into the mathematical structure of
creation/an\-ni\-hi\-la\-tion operators on (bosonic or symmetric) Fock space.
My aim is two-fold: to introduce an axiomatic setting for commutation
relations and
coherent states, and to provide and exercise one such model of combinatorial
nature.  
In the spirit of Paul Dirac's credo
\begin{quote}
  ``One should allow oneself to be led in the direction which the mathematics
  suggests~\ldots\ one must follow up a mathematical idea and see what its
  consequences are, even though one gets led to a domain which is completely
  foreign to what one started with~\ldots\ Mathematics can lead us in a
  direction we would not take if we only followed up physical ideas by
  themselves.''
%
  %I learned to distrust all physical concepts as the basis for a theory.
  %Instead one should put one's trust in a mathematical scheme, even if the
  %scheme does not appear at first to be connected with physics.  One
  %should concentrate on getting interesting mathematics.
%\hfill Paul Dirac
\end{quote}
my hope is that %, in the spirit of the quote above, 
the mathematical theories presented here, and the ideas that underly them, can
be of use to physics.

%\bigskip
%The development is presented in two sections, respectively devoted to the
%axiomatics and the combinatorial model.

\paragraph{Axiomatics.}

Section~\ref{AxiomaticTheory} considers the axiomatics.  This is set up in the
framework of category theory, which is particularly suitable for our purposes.
Our starting point is the consideration of categories of spaces and linear
maps.  So as to be able to accommodate Fock space, these should allow for the
formation of superposed and of noninteracting systems.  In
Section~\ref{SpacesAndLinearMaps}, I respectively formalise these as
compatible biproduct and symmetric monoidal structures.  The linear-algebraic
structure is then derived by convolution with respect to the biproduct
structure.  For completeness, other equivalent formalisations are also given.
Of central importance to our development is the algebraic axiomatisation of
biproduct structure as monoidal bialgebra structure~(see
Proposition~\ref{AlgebraicBiproductStructure} and
Lemma~\ref{BiproductBialgebraStructure}).  
The resulting setting is rich enough for formalising Fock space together with
creation/annihilation operators on it.  Specifically, in
Section~\ref{BosonicFockSpace}, the Fock-space construction is axiomatised as
a functor on the category of spaces and linear maps that transforms the
biproduct~(\ie~superposition) structure to the symmetric
monoidal~(\ie~noninteracting) structure.  A fundamental aspect of this
definition is that it lifts the biproduct bialgebra structure to a
bialgebra structure on Fock space.  This allows for a general definition of
creation/annihilation operators~(Definition~\ref{CreationAnnihilationDef}) and
embodies the essential mathematical structure of the commutation
relations~(Theorem~\ref{Commutation_Relations}).
Section~\ref{CoherentStates} considers coherent states on Fock space.  To this
end, however, one needs specialise the discussion to Fock-space constructions
with suitable comonad structure.  This additional structure plays two roles:
it provides a canonical notion of annihilation operator and permits the
association of coherent states in Fock space to
vectors~(Definition~\ref{ExtensionsDef}~(\ref{GlobalElementExtension}) and
Theorem~\ref{CoherentStateTheorem}).

\paragraph{Combinatorial model.}

Section~\ref{CombinatorialModel} puts forward a bicategorical combinatorial
model.  Its combinatorial nature resides in the structure being a
generalisation of that of the \emph{combinatorial species of
structures} of Joyal~\cite{JoyalAdv,Joyal1234} %,BLL}, 
(see~\cite{FGHW} for details).  The main consequence of this for us here is
that identities, such as the commutation relations, %(for which more below),
acquire combinatorial meaning in the form of natural bijective
correspondences.

The combinatorial model is based on the bicategory of profunctors (or
bimodules, or distributors) as the setting for spaces and linear maps.  These
structures, I briefly review in Section~\ref{TheBicategoryOfProfunctors}
noting analogies with vector spaces.  Combinatorial (bosonic or symmetric)
Fock space is then introduced in Section~\ref{CombinatorialFockSpace}.  The
definition mimics that of the conventional construction as a biproduct of
symmetric tensor powers.  After making explicit the mathematical structure of
combinatorial Fock space, the commutation relation involving creation and
annihilation is considered.  We see here that the essence of its combinatorial
content arises from the simple %combinatorial 
fact that
$$
\symmgroup_{n+1}
\ \iso\
\symmgroup_n
\cup
\big([n] \times \symmgroup_n\big)
\qquad
\mbox{for $[n]=\setof{1,\ldots,n}$}
$$
classifying the permutations on the set $[n+1]$ according as to whether or
not they fix the element $n+1$, see~(\ref{CombinatorialContentOne})
and~(\ref{CombinatorialContentTwo}).
It is an important aspect of the theory, however, that all such
calculations are done formally in the \emph{calculus of coends} (within
the \emph{generalized logic} of Lawvere~\cite{LawvereGenMet}).  I further
illustrate how the calculus can be seen diagrammatically.

Finally, Section~\ref{CombinatorialCoherentStates} considers coherent states
in the combinatorial model.  Taking advantage of the duality structure
available in it, a notion of exponential (in the form of a comonadic/monadic
convolution) is introduced.  The exponential of the creation operator of a
vector at the vacuum state is shown, both algebraically and
combinatorially, to yield the coherent state of the vector.

\paragraph{Related work.}
This work lies at the intersection of computer science, logic, mathematics,
and physics.  As such, it bears relationship with a variety of developments.

%%%Linear logic
%
In relation to mathematical logic, the notion of comonad needed in the
discussion of coherent states is as it arises in models of the
\emph{linear logic} of Girard~\cite{GirardLL}.  The connection between the
exponential modality of linear logic and the Fock-space construction of
physics was recognised long ago by Panangaden
(see~\eg~\cite{BlutePanangadenSeely,BlutePanangaden}).  In view of recent
developments, however, the connection further puts this work in the
context of models of the \emph{differential linear logic}
of Ehrhard and Regnier~\cite{EhrhardRegnierDiff}; and indeed the models to
be found in~\cite{EhrhardKSS,EhrhardFS,BCS,Hyvernat,BluteEhrhardTasson}
all fall within the axiomatisation here.  A stronger axiomatisation (of
which the combinatorial model is the motivating
example~\cite{FioreGSdraft}) leading to fully-fledged differential
structure has been pursued in~\cite{FioreDiff}.

%%% Axiomatics
%
An axiomatics for Fock space has independently been considered
by Vicary~\cite{Vicary}.  His setting, which aims at a tight
correspondence with that of Fock space on Hilbert space, is stronger than
the minimalist one put forward here.  As acknowledged in his work, the
argument used for establishing the commutation relation between creation
and annihilation is based on a private communication of mine.  

%%% Physics
%
The combinatorial model is closely related to the \emph{stuff-type model}
of Baez and Dolan~\cite{BaezDolan}, see also~\cite{Morton}, being both
founded on species of structures.  Roughly, their main difference resides
in that the combinatorial model organises structure as presheaves, whilst
the stuff-type model does so as bundles.  
%A benefit of the former is the calculus of generalised logic.

In connection to mathematical physics, the stuff-type model has been
related to Feynman diagrams and, in connection to mathematical logic,
these have been related to %, proof nets,
the proof theory of linear logic by means of the \mbox{\emph{$\phi$-calculus}}
of Blute and Panangaden~\cite{BlutePanangaden}, which, in turn, has formal
syntactic structure similar to that of the calculus of the combinatorial
model.  These intriguing relationships are worth investigating.

\paragraph{Acknowledgements.}

The mathematical structure underlying the combinatorial model in the
setting of \emph{generalised species of structures} was developed in
collaboration with Nicola Gambino, Martin Hyland, and Glynn
Winskel~\cite{FioreFOSSACS,FGHW,HylandGen}.  
The fact that %That 
it supports creation/annihilation operators, I~realised shortly after
giving a \href{http://www.maths.ox.ac.uk/node/4340/}{seminar at Oxford in
2004} on this material and the differential structure of generalised
species of structures~\cite{FioreGSdraft,FioreFOSSACS,FioreSGDP} where
Prakash Panangaden raised the %interesting 
question.
The axiomatics came later~\cite{FioreDiff}, and was influenced by the work
of Thomas Ehrhard and Laurent Regnier on differential
nets~\cite{EhrhardRegnierDiffNets}. 
%and shaped by the combinatorial model.
%
The work presented here is a write up of the talk~\cite{FioreCQL}, which I
was invited to give by Bob Coecke.
I'm grateful to them all for their part in this work.

\section{Axiomatic theory}
\label{AxiomaticTheory}

This section introduces an axiomatisation of the (bosonic or symmetric)
Fock-space construction on categories of spaces and linear maps,
see~\eg~\cite{Geroch}.

Spaces and linear maps are axiomatised by means of a category~$\Spaces$
equipped with compatible biproduct~$(\O,\biprod)$ and symmetric
monoidal~$(\I,\tensor)$ structures.  
Section~\ref{SpacesAndLinearMaps} reviews these notions and explains the
linear-algebraic structure that they embody.
For a category of spaces and linear maps, the Fock-space construction is
axiomatised as a strong symmetric monoidal functor~$\Fock$ mapping
$(\O,\biprod)$ to $(\I,\tensor)$.  
Section~\ref{BosonicFockSpace} reviews this notion and explains how it
supports an axiomatisation of creation/annihilation operators subject to
commutation relations.
For $\Fock$ underlying a linear exponential comonad, coherent states are
considered and studied in Section~\ref{CoherentStates}.

\subsection{Spaces and linear maps}
\label{SpacesAndLinearMaps}

\paragraph{Biproduct structure.}

A category with finite coproducts and finite products is said to be
\emph{bicartesian}.  One typically writes $0, +$ for the empty and binary
coproducts and $1, \times$ for the empty and binary products.

An object that is both initial and terminal~(\ie~an empty coproduct and
product) is said to be a \emph{zero object}.  For a zero object~$\O$, I will
write $\O_{A,B}$ for the map $A\rightarrow B$ given by the composite
$A\rightarrow \O\rightarrow B$.  

\begin{definition}
A bicartesian category is said to have \emph{biproducts} whenever:
\begin{enumerate}
  \item it has a zero object $\O$, and 
  \item for all objects $A$ and $B$, the canonical map
    $$
    \big[\pair{\id_A,\O_{A,B}},\pair{\O_{B,A},\id_B}\big]
      :A+B\rightarrow A\times B
    $$
    is an isomorphism.
\end{enumerate}
In this context, one typically writes $\biprod$ for the binary biproduct.
\end{definition}

The proposition below gives an algebraic presentation of biproduct
structure which is crucial to our development.
Recall that a \emph{symmetric monoidal structure}
$(\I,\tensor,\lambda,\rho,\alpha,\symm)$ on a category $\lscat C$ is given
by an object $\I\in\lscat C$, a functor $\tensor:\lscat
C^2\rightarrow\lscat C$, and natural isomorphisms $\lambda_C:\I\tensor
C\iso C$, $\rho_C:C\tensor\I\iso C$, $\alpha_{A,B,C}:(A\tensor B)\tensor
C\iso A\tensor(B\tensor C)$, and $\symm_{A,B}: A\tensor B\iso B\tensor A$
subject to coherence conditions, see~\eg~\cite{MacLane}.
\hide{%%% BEGIN HIDE !!!
the coherence conditions
$$\xymatrix{
\ar[dr]_-{\rho_A\tensor\id_B}
(A\tensor\I)\tensor B \ar[rr]^-{\alpha_{A,\I,B}} & & A\tensor(\I\tensor B)
\ar[dl]^-{\id_A\tensor\lambda_B}
\\
& A\tensor B& 
}
$$
$$\xymatrix@C45pt{
\ar[d]|-{\alpha_{A,B,C}\tensor\id_D} 
\big((A\tensor B)\tensor C\big)\tensor D \ar[r]^-{\alpha_{A\tensor B,C,D}}
&
(A\tensor B)\tensor (C\tensor D) \ar[r]^-{\alpha_{A,B,C\tensor D}} &
A\tensor \big(B\tensor (C\tensor D)\big) 
\ar[d]|-{\id_A\tensor\alpha_{B,C,D}}
\\
\big(A\tensor (B\tensor C)\big)\tensor D \ar[rr]_-{\alpha_{A,B\tensor C,D}} & &
A\tensor \big((B\tensor C)\tensor D)\big) 
}$$
$$\xymatrix{
& B\tensor A \ar[dr]^-{\sigma_{B,A}} & 
\\
\ar[ur]^-{\sigma_{A,B}}
A\tensor B \ar[rr]_{\id_{A\tensor B}} & & 
A\tensor B
}
$$
$$\xymatrix{
& A\tensor(B\tensor C) \ar[r]^-{\sigma_{A,B\tensor C}}
&
(B\tensor C)\tensor A
\ar[dr]^-{\alpha_{B,C,A}}
&
\\
(A\tensor B)\tensor C 
\ar[ru]^-{\alpha_{A,B,C}}
\ar[dr]_-{\sigma_{A,B}\tensor\id_C}
& & & B\tensor(C\tensor A)
\\
& (B\tensor A)\tensor C 
\ar[r]_-{\alpha_{B,A,C}} & 
B\tensor(A\tensor C) 
\ar[ru]_-{\id_B\tensor\sigma_{A,C}} &
}$$
}%%% END HIDE !!!

\begin{proposition}\label{AlgebraicBiproductStructure}
To give a choice of biproducts in a category is equivalent to giving a
symmetric monoidal structure $(\O,\biprod)$ on it together with natural
transformations 
\begin{equation}\label{BiproductBialgebraData}
\hfill
\begin{minipage}{0cm}
\xymatrix@C=20pt@R=5pt{
\Zero \ar[dr]^-{\initialmap_A} & & \Zero
\\
& A \ar[dr]_-{\Diag_A} \ar[ur]^-{\terminalmap_A} &
\\
A\biprod A \ar[ru]_-{\coDiag_A} & & A \biprod A
}
\end{minipage}
\hfill\hfill
\end{equation}
such that 
\begin{enumerate}
  \item
    $(A,\initialmap_A,\coDiag_A)$ is a commutative monoid.

\begin{equation}\label{BiproductCommutativeMonoidOne}
\hspace{-35mm}
\begin{minipage}{10cm}
$\begin{array}{c}
\xymatrix@C35pt{
\Zero\biprod\IdX \ar[r]^-{\initialmap_\IdX\biprod\id_\IdX} \ar[dr]|-\iso &
\IdX\biprod\IdX \ar[d]|-{\coDiag_\IdX} & \IdX\biprod\Zero
\ar[l]_-{\id_\IdX\biprod\initialmap_\IdX} \ar[dl]|-\iso
\\
& \IdX & 
}
\
\xymatrix@C45pt{
\IdX\biprod\IdX\biprod\IdX
\ar[d]_-{\coDiag_\IdX\biprod\id_\IdX}
\ar[r]^-{\id_\IdX\biprod\coDiag_\IdX} 
& \IdX\biprod\IdX \ar[d]^-{\coDiag_\IdX}
\\ 
\IdX\biprod\IdX \ar[r]_-{\coDiag_\IdX} & \IdX
}
\end{array}$
\end{minipage}
\end{equation}
\begin{equation}\label{BiproductCommutativeMonoidTwo}
\mbox{}\hfill
\begin{minipage}{3cm}
\xymatrix{
\IdX\biprod\IdX \ar[rr]^-{\symm_{A,A}} \ar[dr]_-{\coDiag_A} & & \IdX\biprod\IdX
\ar[dl]^-{\coDiag_A}
\\ 
& \IdX &
}
\end{minipage}
\hfill\hfill\mbox{}
\end{equation}

  \item
    $(A,\terminalmap_A,\Diag_A)$ is a commutative comonoid.  

\begin{equation}\label{BiproductCommutativeCoMonoidOne}
\hspace{-85mm}
\begin{minipage}{5cm}
$\begin{array}{c}
\xymatrix@C35pt{
& \ar[dl]|-\iso \IdX \ar[d]|-{\Diag_\IdX} \ar[dr]|-\iso & 
\\
\Zero\biprod\IdX & \ar[l]^-{\terminalmap_\IdX\biprod\id_\IdX}
\IdX\biprod\IdX
\ar[r]_-{\id_\IdX\biprod\terminalmap_\IdX} & \IdX\biprod\Zero
}
\
\xymatrix@C45pt{
\IdX \ar[r]^-{\Diag_\IdX} \ar[d]_-{\Diag_\IdX} &
\IdX\biprod\IdX
\ar[d]^-{\id_\IdX\biprod\Diag_\IdX} 
\\
\IdX\biprod\IdX \ar[r]_-{\Diag_\IdX\biprod\id_\IdX} &
\IdX\biprod\IdX\biprod\IdX
&
}
\end{array}$
\end{minipage}
\end{equation}
\begin{equation}\label{BiproductCommutativeCoMonoidTwo}
\mbox{}\hfill
\begin{minipage}{3cm}
\xymatrix{
& \ar[dl]_-{\Diag_\IdX} \IdX \ar[dr]^-{\Diag_\IdX} & 
\\
\IdX\biprod\IdX \ar[rr]_-{\symm_{\IdX,\IdX}} & & \IdX\biprod\IdX
}
\end{minipage}
\hfill\hfill\mbox{}
\end{equation}

  \item
    $\begin{array}[t]{lcl}
      \initialmap_{A\biprod B} & = &
      (\, \O \iso \O\biprod\O
      \xymatrix@C40pt{\ar[r]^-{\initialmap_A\biprod\initialmap_B}&}
      A\biprod B \,)
      \\[2mm]
      \terminalmap_{A\biprod B} & = &
      (\, A\biprod B 
      \xymatrix@C40pt{\ar[r]^-{\terminalmap_A\biprod\terminalmap_B}&}
      \O\biprod\O \iso \O \,)
      \\[2mm]
      \coDiag_{A\biprod B}
      & = & 
      \big(\, (A\biprod B)\biprod(A\biprod B) \iso 
      (A\biprod A)\biprod(B\biprod B) 
      \xymatrix@C40pt{\ar[r]^-{\coDiag_A\biprod\coDiag_B}&}
      A\biprod B \,\big)
      \\[2mm]
      \Diag_{A\biprod B}
      & = & 
      \big(\, (A\biprod B)
      \xymatrix@C40pt{\ar[r]^-{\Diag_A\biprod\Diag_B}&}
      (A\biprod A)\biprod(B\biprod B) 
      \iso
      (A\biprod B)\biprod(A\biprod B) \,\big)
    \end{array}$
\end{enumerate}
\end{proposition}

The biproduct structure induced by~(\ref{BiproductBialgebraData}) has
coproduct diagrams 
$$\xymatrix@C40pt{
\ar@<1.5ex>@/^1em/[rr]^-{\inj1}
A & \hspace{-15mm}
\iso A\biprod\O \ar[r]^-{\id_A\biprod\initialmap_B} & A\biprod B &
\ar[l]_-{\initialmap_A\biprod\id_B} \O\biprod B\iso \hspace{-15mm} &
B 
\ar@<-1.5ex>@/_1em/[ll]_-{\inj2}
}$$
and product diagrams
$$
\xymatrix@C40pt{
A & 
\hspace{-15mm}
\iso A\biprod\O & 
\ar@<-1.5ex>@/_1em/[ll]_-{\proj1}
\ar[l]_-{\id_A\biprod\terminalmap_B} A\biprod B
\ar[r]^-{\terminalmap_A\biprod\id_B} 
\ar@<1.5ex>@/^1em/[rr]^-{\proj2}
& \O\biprod B\iso \hspace{-15mm} & B
}
$$

\begin{proposition}\label{KroneckerProperty}
In a category with biproduct structure~$(\O,\biprod)$, we have that
$$
(\xymatrix{A \ar[r]^-{\inj i} & A\biprod A \ar[r]^-{\proj j} & A})
=\left\{\begin{array}{ll}
  \id_A & \mbox{, if $i=j$}
  \\
  \O_{A,A} & \mbox{, if $i\not= j$}
\end{array}\right.$$
\end{proposition}

\begin{lemma}\label{BiproductBialgebraStructure}
In a category with biproduct
structure~$(\Zero,\biprod;\initialmap,\coDiag;\terminalmap,\Diag)$,
the commutative monoid and comonoid structures
$(\initialmap,\coDiag;\terminalmap,\Diag)$ form a commutative
bialgebra.  That is, $\initialmap$ and $\coDiag$ are comonoid
homomorphisms and, equivalently, $\terminalmap$ and $\Diag$ are monoid
homomorphisms.
\begin{equation}\label{BiproductBialgebraOne}
\begin{minipage}{4cm}
\xymatrix{
& \IdX \ar[dr]^-{\terminalmap_A} &
\\
\Zero \ar[ru]^-{\initialmap_A} \ar[rr]_-{\id_\O} & & \Zero
}
\end{minipage}
\
\begin{minipage}{4cm}
\xymatrix@R20pt@C35pt{
\IdX\biprod\IdX \ar[r]^-{\coDiag_\IdX} \ar[d]_-{\Diag_\IdX\biprod\Diag_\IdX} &
\IdX \ar[r]^-{\Diag_\IdX} & \IdX\biprod\IdX \\
\IdX\biprod\IdX\biprod\IdX\biprod\IdX
\ar[rr]_-{\id_\IdX\biprod\symm_{\IdX,\IdX}\biprod\id_\IdX} 
& &
\IdX\biprod\IdX\biprod\IdX\biprod\IdX
\ar[u]_-{\coDiag_\IdX\biprod\coDiag_\IdX}
}
\end{minipage}
\end{equation}
\begin{equation}\label{BiproductBialgebraTwo}
\hspace{-45mm}
\begin{minipage}{4cm}
$\begin{array}{c}
\xymatrix@C=10pt@R=15pt{
& \IdX \ar[dr]^-{\Diag_\IdX} &
\\
\Zero \ar[ur]^-{\initialmap_\IdX} \ar[dr]|-\iso & & \IdX\biprod\IdX
\\
& \Zero\biprod\Zero \ar[ru]_-{\initialmap_\IdX\biprod\initialmap_\IdX}&
}
\qquad
\xymatrix@C=10pt@R=15pt{
& \IdX \ar[dr]^-{\terminalmap_\IdX} &
\\
\IdX\biprod\IdX \ar[ur]^-{\coDiag_\IdX}
\ar[dr]_-{\terminalmap_\IdX\biprod\terminalmap_\IdX} & & \Zero
\\
& \Zero\biprod\Zero \ar[ru]|-\iso &
}
\end{array}$
\end{minipage}
\end{equation}
\end{lemma}

\paragraph{Linear-algebraic structure.}

We examine the linear-algebraic structure of categories with biproduct
structure.  This I~present in the language of enriched category
theory~\cite{KellyEnrichedCT}.

Let %To this end, let 
$\Mon$ ($\CMon$) be the symmetric monoidal category of (commutative)
monoids with respect to the universal bilinear tensor product.
Recall that\linebreak \mbox{$\Mon$-categories} ($\CMon$-categories)
are categories all of whose homs~$[A,B]$ come equipped with a
(commutative) monoid structure 
$$
\z_{A,B}\in[A,B]
\enspace ,\quad 
\plus_{A,B}:[A,B]^2\rightarrow[A,B]
$$ 
such that composition is strict and bilinear; that is,
\begin{center}
$\z_{B,C}\icomp f = \z_{A,C}$
\qquad and \qquad
$f \icomp \z_{C,A} = \z_{C,B}$
\end{center}
for all $f: A\rightarrow B$, and
\begin{center}
  $g \icomp (f\plus_{A,B} f') = g\icomp f \plus_{A,C} g\icomp f'$
\qquad and \qquad
$(g\plus_{B,C} g')\icomp f = g\icomp f \plus_{A,C} g'\icomp f$
\end{center}
for all $f,f':A\rightarrow B$ and $g,g':B\rightarrow C$.

\begin{proposition}\label{Mon-Enrichment}
The following are equivalent.
\begin{enumerate}
\item
Categories with biproduct structure.

\item
$\Mon$-categories with \emph{(}necessarily enriched\emph{)} finite
products.

\item
$\CMon$-categories with \emph{(}necessarily enriched\emph{)} finite
products.
\end{enumerate}
\end{proposition}
The enrichment of categories with biproduct
structure~$(\O,\biprod;\initialmap,\coDiag;\terminalmap,\Diag)$ is given
by \emph{convolution}~(see~\eg~\cite{Sweedler}) as follows:
$$\begin{array}{l}
\z_{A,B}
\ = \
\xymatrix{
  ( A \ar[r]^-{\terminalmap_A} & \Zero \ar[r]^-{\initialmap_B} & B ) }
\ = \
\O_{A,B}
\\
f+_{A,B}g
\ = \
\xymatrix{
  ( A \ar[r]^-{\Diag_A} & A\biprod A \ar[r]^-{f\biprod g} & B\biprod B
  \ar[r]^-{\coDiag_B} & B )
}
\end{array}$$

\begin{proposition}
In a category with biproduct structure, $\coDiag_A = \proj1+\proj2:A\biprod
A\rightarrow A$ and $\Diag_A = \inj1+\inj2:A\rightarrow A\biprod A$.
\end{proposition}

We now consider biproduct structure on symmetric monoidal categories.  To this
end, note that in a monoidal category with tensor~$\tensor$ and binary
products~$\product$ there is a natural distributive law as follows:
\[%\begin{equation*}\label{Distributive_Law}
\dlaw_{A,B,C} = \pair{\proj1\tensor\id_C,\proj2\tensor\id_C}:
  (\X\product\Y)\tensor\Z \rightarrow (\X\tensor\Z)\product (\Y\tensor\Z)
%\enspace.
\]%\end{equation*}

\begin{definition}
A biproduct structure~$(\Zero,\biprod;\initialmap,\coDiag;\terminalmap,\Diag)$
and a symmetric monoidal structure~$(\tensorunit,\tensor)$ on a category are
\Def{compatible} whenever the following hold:
\[%\begin{equation*}\label{Tensor_Strictness_Diagram}
%\begin{minipage}{\textwidth}
\xymatrix@C=10pt@R=15pt{
& \Zero\tensor\Z \ar[dr]^-{\initialmap_A\tensor\id}
\ar[dd]|-{\terminalmap_{\O\tensor C}} & 
\\
\X\tensor\Z\ar[ur]^{\terminalmap_A\tensor\id_C}
\ar[dr]_-{\terminalmap_{A\tensor C}} & &
\X\tensor\Z
\\
& \Zero \ar[ur]_-{\initialmap_{A\tensor C}} &
}
%\end{minipage}
%\end{equation*}
%%%
\qquad\quad
%%%
%\begin{equation*}\label{Tensor_Bilinearity_Diagram}
%\begin{minipage}{\textwidth}
\xymatrix@C=10pt@R=15pt{
& (\X\biprod\X)\tensor\Z \ar[dd]^-{\dlaw_{A,A,C}}
\ar[dr]^-{\coDiag_A\tensor\id_C} & \\
\X\tensor\Z \ar[ur]^-{\Diag_A\tensor\id} \ar[dr]_(.4){\Diag_{A\tensor C}} &  &
\X\tensor\Z
\\
& (\X\tensor\Z) \biprod (\X\tensor\Z) \ar[ur]_(.6){\coDiag_{A\tensor C}} &
}
%\end{minipage}
\]%\end{equation*}
\end{definition}

Proposition~\ref{Mon-Enrichment} extends to the symmetric monoidal setting.
Recall that a
\Def{$\Mon$-enriched \emph{(}symmetric\emph{)} monoidal category} is a
(symmetric) monoidal category with a $\Mon$-enrichment for which the
tensor is strict and bilinear; that is, such that
\[%\begin{equation*}\label{Tensor_Strictness}
\z_{X,Y}\tensor f = \z_{X\tensor A,Y\tensor B}
\qquad \mbox{and} \qquad
f \tensor \z_{X,Y} = \z_{A\tensor X,B\tensor Y}
\]%\end{equation*}
for all $f: A\rightarrow B$, and
\[%\begin{equation*}\label{Tensor_Bilinearity}
g \tensor (f\plus f') = g\tensor f \plus g\tensor f'
\qquad \mbox{and} \qquad
(g\plus g')\tensor f = g\tensor f \plus g'\tensor f
\]%\end{equation*}
for all $f,f':A\rightarrow B$ and $g,g':X\rightarrow Y$.

\begin{proposition}
The following are equivalent.
\begin{enumerate}
\item
Categories with compatible biproduct and symmetric monoidal structures.

\item
$\Mon$-enriched symmetric monoidal categories with \emph{(}necessarily
enriched\emph{)} finite products.

\item
$\CMon$-enriched symmetric monoidal categories with \emph{(}necessarily
enriched\emph{)} finite products.
\end{enumerate}
\end{proposition}

\begin{definition}
A category with compatible biproduct and symmetric monoidal structures is
referred to as a \emph{category of spaces and linear maps}.
\end{definition}

\subsection{Fock space}
\label{BosonicFockSpace}

\paragraph{Strong-monoidal functorial structure.}

A \emph{strong monoidal functor}~$(F,\phi,\varphi):(\lscat
C,\I,\tensor)\rightarrow(\lscat C',\I',\tensor')$ between monoidal categories
consists of a functor~$F:\lscat C\rightarrow \lscat C'$, an isomorphism
$\phi:\I'\iso F(\I)$, and a natural isomorphism $\varphi_{A,B}:FA\tensor'
FB\iso F(A\tensor B)$ subject to the coherence conditions below.
$$
\xymatrix@C35pt{
\ar[d]_-{\id_{FC}\tensor'\phi}
FC\tensor'\I' \ar[r]^-{\rho'_{FC}} & FC
\\
FC\tensor'F\I \ar[r]_-{\varphi_{C,\I}} & F(C\tensor\I) \ar[u]_-{F\rho_C}
}
\qquad
\xymatrix@C35pt{
\ar[d]_-{\phi\tensor'\id_{FC}}
\I' \tensor' FC \ar[r]^-{\lambda'_{FC}} & FC
\\
F\I \tensor' FC \ar[r]_-{\varphi_{\I,C}} & F(\I\tensor C) \ar[u]_-{F\lambda_C}
}
$$
$$\xymatrix@C45pt{
\ar[d]_-{\varphi_{A,B}\tensor'\id_{FC}}
(F A\tensor' FB)\tensor' FC \ar[r]^-{\alpha_{FA,FB,FC}} 
& F A\tensor' (FB\tensor' FC) \ar[r]^-{\id_{FA}\tensor'\varphi_{A,B}}
& F A\tensor' F(B\tensor C)
\ar[d]^-{\varphi_{A,B\tensor C}}
\\
F (A\tensor B)\tensor' FC \ar[r]_-{\varphi_{A\tensor B},C} &
F \big((A\tensor B)\tensor C\big) \ar[r]_-{F\alpha_{A,B,C}} &
F \big(A\tensor (B\tensor C)\big) 
& 
}$$
$$\xymatrix@C35pt{
\ar[d]_-{\sigma_{FA,FB}}
FA\tensor'FB \ar[r]^-{\varphi_{A,B}} & F(A\tensor B) \ar[d]^-{F(\sigma_{A,B})}
\\
FB\tensor'FA \ar[r]_-{\varphi_{B,A}} & F(B\tensor A) 
}$$

\begin{definition}
A strong monoidal
functor~$(\Spaces,\O,\biprod)\rightarrow(\Spaces,\I,\tensor)$ for a category
of spaces and linear maps~$\Spaces$ is referred to as a \emph{(bosonic or
symmetric) Fock-space construction}.  
\end{definition}

The Fock-space construction supports operations for \emph{initialising}
and \emph{merging}~$(\banginitialmap,\bangcoDiag)$, and for
\emph{finalising} and \emph{splitting}~$(\bangterminalmap,\bangDiag)$.

\begin{definition}
For a Fock-space construction on a category of spaces and linear maps, set:
$$\begin{array}{ll}
\banginitialmap_\IdX = 
( \xymatrix{\tensorunit \iso \bang\Zero \ar[r]^-{\bang\initialmap_\IdX} & \bang\IdX} )
\enspace , & \
\bangcoDiag_\IdX = 
( \xymatrix@C=10pt{\bang\IdX\tensor\bang\IdX \iso
\bang(\IdX\biprod\IdX)\ar[rr]^-{\bang\coDiag_\IdX} && \bang\IdX })
\\
\bangterminalmap_\IdX  = 
( \xymatrix{\Fock A \ar[r]^-{\Fock\terminalmap_\IdX} & \Fock\O \iso \I} )
\enspace , & \
\bangDiag_\IdX =  
( \xymatrix{\Fock A \ar[r]^-{\Fock\Diag_\IdX} & \Fock(A\biprod A)\iso\Fock
A\tensor\Fock A})
\end{array}$$
\hide{%%% BEGIN HIDE !!!
$$\begin{array}{rcl}
\banginitialmap_\IdX & = & 
( \xymatrix{\tensorunit \iso \bang\Zero \ar[r]^-{\bang\initialmap_\IdX} & \bang\IdX} )
\\
\bangcoDiag_\IdX & = & 
( \xymatrix@C=10pt{\bang\IdX\tensor\bang\IdX \iso
\bang(\IdX\biprod\IdX)\ar[rr]^-{\bang\coDiag_\IdX} && \bang\IdX }
)
\\[2mm]
\bangterminalmap_\IdX & = & 
( \xymatrix{\Fock A \ar[r]^-{\Fock\terminalmap_\IdX} & \Fock\O \iso \I} )
\\
\bangDiag_\IdX & = & 
( \xymatrix{\Fock A \ar[r]^-{\Fock\Diag_\IdX} & \Fock(A\biprod A)\iso\Fock
A\tensor\Fock A})
%\\[2mm]
%\q i & = & 
%( \xymatrix{\Fock A\ar[r]^-{\Fock\inj i}&\Fock(A\biprod A)\iso\Fock
%  A\tensor\Fock A}
%  )
%\\
%\p i & = & 
%( \xymatrix{\Fock A\tensor\Fock A\iso\Fock(A\biprod A)\ar[r]^-{\Fock\proj i}
%  & \Fock A } )
\end{array}$$
}%%% END HIDE !!!
\end{definition}

The commutative bialgebra structure induced by the biproduct structure
yields commutative bialgebraic structure on Fock space.

\begin{lemma}
For a Fock-space construction~$\Fock$ on a category of spaces and linear
maps, the natural transformations
\begin{equation}\label{FockSpaceCommutativeBialgebraData}
\hfill
\begin{minipage}{4cm}\xymatrix@C=20pt@R=5pt{
\tensorunit \ar[dr]^-{\banginitialmap_\IdX} & & \tensorunit
\\
& \bang\IdX \ar[dr]_-{\bangDiag_\IdX} \ar[ur]^-{\bangterminalmap_\IdX} & 
\\
\bang\IdX\tensor\bang\IdX \ar[ru]_-{\bangcoDiag_\IdX} & &
\bang\IdX\tensor\bang\IdX
}\end{minipage}
\hfill\hfill
\end{equation}
form a commutative bialgebra.
\end{lemma}
Indeed, by means of the coherence conditions of strong monoidal functors,
the application of $\Fock$ to the
diagrams~(\ref{BiproductCommutativeMonoidOne}--%), 
%(\ref{BiproductCommutativeMonoidTwo}),
%(\ref{BiproductCommutativeCoMonoidOne}),
%(\ref{BiproductCommutativeCoMonoidTwo}),
%(\ref{BiproductBialgebraOne}), 
%(
\ref{BiproductBialgebraTwo}) yields the commutativity of the diagrams
below.
$$
\xymatrix@C35pt{
\I\tensor\Fock A \ar[r]^-{\banginitialmap_A\tensor\id_{\Fock A}}
\ar[dr]|-\iso & \Fock A\tensor\Fock A \ar[d]|-{\bangcoDiag_A} & \Fock
A\tensor\I \ar[l]_-{\id_{\Fock A}\tensor\banginitialmap_A} \ar[dl]|-\iso
\\
& \Fock A & 
}
\enspace
\xymatrix@C45pt{
\Fock A\tensor\Fock A\tensor\Fock A
\ar[d]_-{\bangcoDiag_A\tensor\id_{\Fock A}}
\ar[r]^-{\id_{\Fock A}\tensor\bangcoDiag_A} 
& \Fock A\tensor\Fock A \ar[d]^-{\bangcoDiag_A}
\\ 
\Fock A\tensor\Fock A \ar[r]_-{\bangcoDiag_A} & \Fock A
}
$$
$$
\xymatrix{
\Fock A\tensor\Fock A \ar[rr]^-{\symm_{\Fock A,\Fock A}}
\ar[dr]_-{\bangcoDiag_A} & & \Fock A\tensor\Fock A
\ar[dl]^-{\bangcoDiag_A} \\ & \Fock A &
}
\quad
\xymatrix{
& \ar[dl]_-{\bangDiag_A} \Fock A \ar[dr]^-{\bangDiag_A} & 
\\
\Fock A\tensor\Fock A \ar[rr]_-{\symm_{\Fock A,\Fock A}} & & \Fock
A\tensor\Fock A
}
$$
$$
\xymatrix@C35pt{
& \ar[dl]|-\iso \Fock A \ar[d]|-{\bangDiag_A} \ar[dr]|-\iso & 
\\
\I\tensor\Fock A & \ar[l]^-{\bangterminalmap_{\Fock A}\tensor\id_{\Fock A}}
\Fock A\tensor\Fock A
\ar[r]_-{\id_{\Fock A}\tensor\bangterminalmap_A} & \Fock A\tensor\I
}
\qquad
\xymatrix@C45pt{
\IdX \ar[r]^-{\Diag_\IdX} \ar[d]_-{\Diag_\IdX} &
\IdX\biprod\IdX
\ar[d]^-{\id_\IdX\biprod\Diag_\IdX} 
\\
\IdX\biprod\IdX \ar[r]_-{\Diag_\IdX\biprod\id_\IdX} &
\IdX\biprod\IdX\biprod\IdX
&
}
$$
$$
\begin{minipage}{4cm}
\xymatrix{
& \Fock A\ar[dr]^-{\bangterminalmap_A} &
\\
\I \ar[ru]^-{\banginitialmap_A} \ar[rr]_-{\id_\I} & & \I
}
\end{minipage}
\enspace
\begin{minipage}{4cm}
\xymatrix@R20pt@C35pt{
\Fock A\tensor\Fock A \ar[r]^-{\bangcoDiag_A}
\ar[d]_-{\bangDiag_A\tensor\bangDiag_A} &
\Fock A\ar[r]^-{\bangDiag_A} & \Fock A\tensor\Fock A \\
\Fock A\tensor\Fock A\tensor\Fock A\tensor\Fock A
\ar[rr]_-{\id_{\Fock A}\tensor\symm_{\Fock A,\Fock A}\tensor\id_{\Fock A}} 
& &
\Fock A\tensor\Fock A\tensor\Fock A\tensor\Fock A
\ar[u]_-{\bangcoDiag_A\tensor\bangcoDiag_A}
}
\end{minipage}
$$
$$
\xymatrix@C=10pt@R=15pt{
& \Fock A \ar[dr]^-{\bangDiag_A} &
\\
\I \ar[ur]^-{\banginitialmap_A} \ar[dr]|-\iso & & \Fock A\tensor\Fock A
\\
& \I\tensor\I \ar[ru]_-{\banginitialmap_A\tensor\banginitialmap_A}&
}
\qquad\qquad
\xymatrix@C=10pt@R=15pt{
& \Fock A \ar[dr]^-{\bangterminalmap_A} &
\\
\Fock A\tensor\Fock A \ar[ur]^-{\bangcoDiag_A}
\ar[dr]_-{\bangterminalmap_A\tensor\bangterminalmap_A} & & \I
\\
& \I\tensor\I \ar[ru]|-\iso &
}
$$

\hide{
\begin{proposition}
For a Fock-space construction~$\Fock$,
\begin{enumerate}
\item 
  $\q1 = (\xymatrix@C45pt{\Fock A\iso\Fock A\tensor\I \ar[r]^-{\id_{\Fock
   A}\tensor\banginitialmap_A}&\Fock A\tensor\Fock A})$
  and 
  $\q2 = (\xymatrix@C45pt{\Fock A\iso\I\tensor\Fock A
  \ar[r]^-{\banginitialmap_A\tensor\id_{\Fock A}}&\Fock A\tensor\Fock A})$

\item
  $\p1 = (\xymatrix@C45pt{\Fock A\tensor\Fock A\ar[r]^-{\id_{\Fock
   A}\tensor\bangterminalmap_A}&\Fock A\tensor\I\iso\Fock A})$
  and 
  $\p2 = (\xymatrix@C45pt{\Fock A\tensor\Fock A
   \ar[r]^-{\bangterminalmap_A\tensor\id_{\Fock A}}&\I\tensor\Fock
   A\iso\Fock A})$

\item
$(\xymatrix{\Fock A \ar[r]^-{\q i} & \Fock A\biprod\Fock A \ar[r]^-{\p j} &
\Fock A})
=\left\{\begin{array}{ll}
  \id_{\Fock A} & \mbox{, if $i=j$}
  \\
  \banginitialmap_A \icomp \bangterminalmap_A & \mbox{, if $i\not= j$}
\end{array}\right.$

\item
  $(\id_{\Fock A}\tensor\symm_{\Fock A,\Fock
   A})\icomp(\q1\tensor\id_{\Fock A})=\id_{\Fock A}\tensor\q1$
  and 
  $(\id_{\Fock A}\tensor\symm_{\Fock A,\Fock
   A})\icomp(\q2\tensor\id_{\Fock A})=(\q2\tensor\id_{\Fock
   A})\icomp\symm_{\Fock A,\Fock A}$.

\item
  $\bangcoDiag_A\icomp\q i=\id_{\Fock A}$ and 
  $\p i\icomp\bangDiag_A=\id_{\Fock A}$.

%\item
%  For $i\not=j$, $\symm_{\Fock A,\Fock A}\icomp\q i = \q j$ and
%  $\p i\icomp\symm_{\Fock A,\Fock A} = \p j$.
\end{enumerate}
\end{proposition}
}

\begin{proposition}\label{varphiAAinverse}
For a Fock-space construction~$(\Fock,\phi,\varphi)$, the isomorphism
$\varphi_{A,B}$ has inverse 
$(\Fock\proj1\tensor\Fock\proj2)\icomp\bangDiag_{A\biprod B}$.
\end{proposition}
\begin{proof}
Follows from the commutativity of %the diagram:
$$\xymatrix@C40pt{
& 
\Fock(A\biprod B)\tensor\Fock(A\biprod B)
\ar[r]^-{\Fock\proj1\tensor\Fock\proj2}
\ar[d]|-\iso^-{\varphi_{A\biprod B,A\biprod B}}
& 
\Fock A\tensor\Fock B
\ar[d]|-\iso^-{\varphi_{A,B}} 
&
\\
\ar@<-1.5ex>@/_2em/[rr]_-{\id_{\Fock(A\biprod B)}}
\ar[ru]^-{\bangDiag_{A\biprod B}}
\Fock(A\biprod B)
\ar[r]_-{\Fock\Diag_{A\biprod B}}
& 
\Fock\big((A\biprod B)\biprod(A\biprod B)\big)
\ar[r]_-{\Fock(\proj1\biprod\proj2)}
&
\Fock(A\biprod B)
}$$\\[-11mm]
\end{proof}

\begin{proposition}
For a Fock-space construction $\Fock$, we have that 
$\Fock(0_{A,B})= \banginitialmap_B\icomp\bangterminalmap_A$ and that
$\Fock(f+g) 
 = \bangcoDiag_B\icomp(\Fock f\tensor\Fock g)\icomp \bangDiag_A
 : \Fock A\rightarrow\Fock B$ for all $f,g:A\rightarrow B$.
\end{proposition}

\subsubsection{Creation/annihilation operators}

\begin{definition}\label{CreationAnnihilationDef}
Let $\Fock$ be a Fock-space construction.  For natural transformations
$\eta_A: A\rightarrow\Fock A$ and $\varepsilon_A:\Fock A\rightarrow A$,
define the associated \emph{creation} (or \emph{raising}) natural
transformation~$\creation\eta$ and \emph{annihilation} (or
\emph{lowering}) natural transformation~$\annihilation\varepsilon$ as
$$\begin{array}{lcl}
\creation\eta_A & = & 
(\xymatrix@C25pt{ A\tensor\Fock A \ar[rr]^-{\eta_A\tensor\id_{\Fock A}} && \Fock
A\tensor \Fock A \ar[r]^-{\bangcoDiag_A} & \Fock A
})
\\[3mm]
\annihilation\varepsilon_A & = & 
(\xymatrix@C25pt{ \Fock A \ar[r]^-{\bangDiag_A} & 
\Fock A\tensor \Fock A \ar[rr]^-{\varepsilon_A\tensor\id_{\Fock A}} &&
A\tensor\Fock A })
\end{array}$$
\end{definition}

The above form for creation and annihilation operators is non-standard.  More
commonly, see~\eg~\cite{Geroch}, the literature deals with creation operators
$\creation\eta_A^v: \bang A\rightarrow\bang A$ for vectors $v:
\tensorunit\rightarrow A$ and annihilation operators
$\annihilation\varepsilon_A^{v'}:\bang A\rightarrow \bang A$ for covectors
$v':A \rightarrow\tensorunit$. %(see~\eg~\cite{Geroch}).  
In the present %abstract 
setting, these are derived as follows: 
$$\begin{array}{rcl}
\creation\eta_A^v 
& \ = \ & 
(\xymatrix@C20pt{\bang A\iso \tensorunit\tensor \bang A
\ar[rr]^-{v\tensor\id_{\Fock A}} && A\tensor\bang A \ar[r]^-{\creation\eta_A} &
\bang A})
\\[1mm]
\annihilation\varepsilon_{A}^{v'}
& = & 
(\xymatrix{\bang A \ar[r]^-{\annihilation\varepsilon_A} & A\tensor\bang
A\ar[rr]^-{v'\tensor\id_{\Fock A}} && \tensorunit\tensor\bang A
\iso \bang A})
\end{array}$$

\begin{theorem}\label{Commutation_Relations}
Let $\Fock$ be a Fock-space construction on a category of spaces and
linear maps.  For natural transformations $\eta_A:A\rightarrow\Fock A$ and
$\varepsilon_A:\Fock A\rightarrow A$, their associated creation and
annihilation natural transformations $\creation\eta_A:A\tensor\Fock
A\rightarrow \Fock A$ and $\annihilation\varepsilon_A:\Fock
A\rightarrow A\tensor\Fock A$ satisfy the commutation relations:
\begin{enumerate}
\item\label{Commutation_Relations_ONE}
$
\annihilation\varepsilon_A\icomp\creation\eta_A
\, = \,
(\varepsilon_A\icomp\eta_A\tensor\id_{\Fock A})
\plus 
(\id_A\tensor\creation\eta_A)(\symm_{A,A}\tensor\id_{\Fock A})(\id_A\tensor\annihilation\varepsilon_A)
\ : A\tensor\bang A\rightarrow A\tensor\bang A
$

\item\label{Commutation_Relations_TWO}
$\creation\eta_A\icomp(\id_A\tensor\creation\eta_A)
\, = \,
\creation\eta_A\icomp(\id_A\tensor\creation\eta_A)\icomp(\symm_{A,A}\tensor\id_{\Fock A})
\ : A\tensor A\tensor\bang A \rightarrow \bang A
$

\item\label{Commutation_Relations_THREE}
$
(\id_A\tensor\annihilation\varepsilon_A)\icomp\annihilation\varepsilon_A
\, = \,
(\symm_{A,A}\tensor\id_{\Fock
A})(\id_A\tensor\annihilation\varepsilon_A)\icomp\annihilation\varepsilon_A \ 
: \bang A \rightarrow A\tensor A\tensor\bang A
$
\end{enumerate}
\end{theorem}
It follows as a corollary that\\[-5mm]
\begin{eqnarray}
\nonumber
\annihilation\varepsilon_A^{v'}\icomp\creation\eta_A^v 
& 
\ = \ &
\big(\xymatrix@C80pt{\bang A\iso\tensorunit\tensor\bang A
\ar[r]^-{(v' \varepsilon_A \eta_A v)\tensor\id_{\Fock A}} & \tensorunit\tensor\bang
A\iso\bang A}\big) 
\\[-2mm]
& & 
\label{VectorCovectorCreationAnnihilation}
\enspace + \
\big(\xymatrix@C35pt{\bang A\ar[r]^-{\creation\eta_A^v\icomp\annihilation\varepsilon_A^{v'}}&\bang
A}\big)
%\ : \bang A\rightarrow\bang A
\\[1mm] \nonumber
\hspace{-20mm}
\creation\eta_A^u\icomp\creation\eta_A^v 
& 
\ = \ & 
\creation\eta_A^v\icomp\creation\eta_A^u
\\[1mm] \nonumber
\hspace{-20mm}
\annihilation\varepsilon_A^{u'}\icomp\annihilation\varepsilon_A^{v'}
& 
\ = \ & 
\annihilation\varepsilon_A^{v'}\icomp\annihilation\varepsilon_A^{u'}
\end{eqnarray}
for all $u,v:\tensorunit\rightarrow A$ and
$u',v':A\rightarrow\tensorunit$.

The proof of the theorem depends on the following lemma.

\begin{lemma}\label{TechnicalLemma}
For a Fock-space construction $\Fock$, the following hold for all natural
transformations $\eta_A:A\rightarrow\Fock A$ and $\varepsilon_A:\Fock
A\rightarrow A$.
\begin{enumerate}
\item\label{TechnicalLemmaOne}
  $\eta_{A\biprod A}\icomp \Diag_A = (\Fock\inj 1+\Fock\inj 2)\icomp\eta_A
   : A\rightarrow\Fock(A\biprod A)$
  and 
  $\coDiag_A \icomp \varepsilon_{A\biprod A}
   = \varepsilon_A \icomp (\Fock\proj 1+\Fock\proj 2)
   : \Fock(A\biprod A)\rightarrow A$.

\item\label{TechnicalLemmaTwo}
  $\bangDiag_A\icomp\eta_A 
   = %(\q1+\q2)\icomp\eta_A : A\rightarrow\Fock A\tensor\Fock A$
   (\xymatrix@C30pt{A\iso
   A\tensor\I\ar[r]^-{\eta_A\tensor\banginitialmap_A}&\Fock A\tensor\Fock
   A}) 
   + 
   (\xymatrix@C30pt{A\iso \I\tensor
   A\ar[r]^-{\banginitialmap_A\tensor\eta_A} &\Fock A\tensor\Fock A})$
  and 
  $\varepsilon_A\icomp\bangcoDiag_A 
   = %\varepsilon_A\icomp(\p1+\p2) : \Fock A\tensor\Fock A\rightarrow A$.
   (\xymatrix@C30pt{\Fock A\tensor\Fock A 
   \ar[r]^-{\varepsilon_A\tensor\bangterminalmap_A}& A\tensor\I \iso A}) 
   + 
   (\xymatrix@C30pt{\Fock A\tensor\Fock A
   \ar[r]^-{\bangterminalmap_A\tensor\varepsilon_A} & \I\tensor A \iso
   A})$.

\item\label{TechnicalLemmaZero}
  $\bangterminalmap_A\icomp\eta_A=0_{A,\I}:A\rightarrow\I$ and
  $\varepsilon_A\icomp\banginitialmap_A=0_{\I,A}:\I\rightarrow A$.
\end{enumerate}
\end{lemma}
\begin{proof}
For the first and third items, I only detail the proof of one of the
identities; the other identity being established dually.

One calculates as follows:

$(\ref{TechnicalLemmaOne})$
$\eta_{A\biprod A}\icomp \Diag_A 
  = \eta_{A\biprod A}\icomp (\inj1+\inj2)
  = \eta_{A\biprod A}\icomp\inj1 + \eta_{A\biprod A}\icomp\inj2
  = \Fock(\inj1)\icomp\eta_A + \Fock(\inj2)\icomp\eta_A  
  = (\Fock\inj1 + \Fock\inj2)\icomp\eta_A$.  

$(\ref{TechnicalLemmaTwo})$
$\begin{array}[t]{rcl}
\bangDiag_A\icomp\eta_A
& = & 
(\Fock\proj1\tensor\Fock\proj2) \icomp \bangDiag_{A\biprod A} \icomp
\Fock(\Diag_A)\icomp \eta_A 
\\
&& \qquad\mbox{, by definition of $\bangDiag$ and
Proposition~\ref{varphiAAinverse}}
\\[1mm]
& = & 
(\Fock\proj1\tensor\Fock\proj2) \icomp \bangDiag_{A\biprod A} \icomp
(\Fock\inj1+\Fock\inj2)\icomp\eta_A
\\
&& \qquad\mbox{, by naturality of $\eta$ and item $(\ref{TechnicalLemmaOne})$
of this lemma} 
\\[1mm]
& = & 
(\Fock\proj1\tensor\Fock\proj2) \icomp
\big((\Fock\inj1\tensor\Fock\inj1)+(\Fock\inj2\tensor\Fock\inj2)\big)
\icomp\bangDiag_A\icomp\eta_A
\\
&& \qquad\mbox{, by naturality of $\bangDiag$}
\\[1mm]
& = & 
\big((\id_{\Fock A}\tensor\banginitialmap_A\bangterminalmap_A)+(\banginitialmap_A\bangterminalmap_A\tensor\id_{\Fock A})\big)
\icomp\bangDiag_A\icomp\eta_A
\\
&& \qquad\mbox{, by Proposition~\ref{KroneckerProperty} and the definitions of
$\banginitialmap$ and $\bangterminalmap$} 
\\[1mm]
& = & 
%(\q1+\q2)\icomp\eta_A
%\\
%&& \qquad\mbox{, by the comonoid structure of $(\bangterminalmap,\bangDiag)$
%and definition of $\q i$} 
(\xymatrix@C30pt{A\iso A\tensor\I\ar[r]^-{\eta_A\tensor\banginitialmap_A}&\Fock
A\tensor\Fock A}) 
+ 
(\xymatrix@C30pt{A\iso \I\tensor A\ar[r]^-{\banginitialmap_A\tensor\eta_A}
&\Fock A\tensor\Fock A}) 
\\
&& \qquad\mbox{, by the comonoid structure of $(\bangterminalmap,\bangDiag)$}
\end{array}$

\medskip

$\begin{array}[t]{rcl}
\varepsilon_A\icomp\bangcoDiag_A
& = & 
(\proj1+\proj2)\icomp\varepsilon_{A\biprod A}\icomp\varphi_{A,A}
\\
&&\qquad\mbox{, by definition of $\bangDiag$ and naturality of
$\varepsilon$}
\\[1mm]
& = & 
(\varepsilon_A\icomp\Fock(\proj1)\icomp\varphi_{A,A})
+
(\varepsilon_A\icomp\Fock(\proj2)\icomp\varphi_{A,A})
\\
&&\qquad\mbox{, by bilinearity of composition and naturality}
\\[1mm]
& = & 
(\xymatrix@C30pt{\Fock A\tensor\Fock A 
   \ar[r]^-{\varepsilon_A\tensor\bangterminalmap_A}& A\tensor\I \iso A}) 
   + 
   (\xymatrix@C30pt{\Fock A\tensor\Fock A
   \ar[r]^-{\bangterminalmap_A\tensor\varepsilon_A} & \I\tensor A \iso
   A})
\\
&&\qquad\mbox{, by definition of $\bangterminalmap$ and coherence of
$\Fock$}
\end{array}$

$(\ref{TechnicalLemmaZero})$
$\bangterminalmap_A\icomp\eta_A
  = (\xymatrix{A\ar[r]^-{\eta_A}&\Fock A
     \ar[r]^-{\Fock\terminalmap_A}&\Fock\O\iso\I}) 
  = (\xymatrix{A\ar[r]^-{\terminalmap_A}&\O\ar[r]^-{\eta_\O}&
     \Fock\O\iso\I})$.
\end{proof}

\begin{proof}[Proof of
Theorem~\ref{Commutation_Relations}]
$(\ref{Commutation_Relations_ONE})$~By means of
Lemma~\ref{TechnicalLemma}~$(\ref{TechnicalLemmaTwo})$, the commutativity
of the diagram
$$\xymatrix{
\ar[dr]^-{\creation\eta_A}
A\tensor\Fock A \ar[d]_-{\eta_A\tensor\id_{\Fock A}}
\ar@<-3.5ex>@/_5em/[dd]|-{(\bangDiag_A\icomp\eta_A)\tensor\bangDiag_A}
& & A\tensor\Fock A
\\
\Fock A\tensor\Fock A \ar[r]^-{\bangcoDiag_A} 
\ar[d]_-{\bangDiag_A\tensor\bangDiag_A}
& 
\ar[ur]^-{\annihilation\varepsilon_A}
\Fock A \ar[r]^-{\bangDiag_A} & \Fock A\tensor\Fock A
\ar[u]_-{\varepsilon_A\tensor\id_{\Fock A}}
\\
\Fock A\tensor\Fock A\tensor\Fock A\tensor\Fock A \ar[rr]_-{\id_{\Fock
A}\tensor\symm_{\Fock A,\Fock A}\tensor\id_{\Fock A}} &&
\Fock A\tensor\Fock A\tensor\Fock A\tensor\Fock A 
\ar[u]_-{\bangcoDiag_A\tensor\bangcoDiag_A}
\ar@<-3.5ex>@/_5em/[uu]|-{(\varepsilon_A\icomp\bangcoDiag_A)\tensor\bangcoDiag_A} 
\\
}$$
shows that $\annihilation\varepsilon_A\icomp\creation\eta_A$ equals
$$
\begin{array}{ll}
& (\xymatrix@C55pt{
A\tensor\Fock A\iso A\tensor\I\tensor\Fock A
\ar[r]^-{\eta_A\tensor\banginitialmap_A\tensor\bangDiag_A}&
\Fock A\tensor\Fock A\tensor\Fock A\tensor\Fock A}
\\
&
\qquad
\xymatrix@C25pt{
\ar[rrr]^-{\id_{\Fock A}\tensor\symm_{\Fock A,\Fock A}\tensor\id_{\Fock A}}
&&&
\Fock A\tensor\Fock A\tensor\Fock A\tensor\Fock A
\ar[rr]^-{\varepsilon_A\tensor\bangterminalmap_A\tensor\bangcoDiag_A}
&&
A\tensor\I\tensor\Fock A \iso A\tensor\Fock A
})
\\
+ & 
\\
& (\xymatrix@C55pt{
A\tensor\Fock A\iso 
\I \tensor A\tensor \Fock A 
\ar[r]^-{ \banginitialmap_A \tensor \eta_A \tensor \bangDiag_A }&
\Fock A\tensor\Fock A\tensor\Fock A\tensor\Fock A}
\\
&
\qquad
\xymatrix@C25pt{
\ar[rrr]^-{\id_{\Fock A}\tensor\symm_{\Fock A,\Fock A}\tensor\id_{\Fock A}}
&&&
\Fock A\tensor\Fock A\tensor\Fock A\tensor\Fock A
\ar[rr]^-{\varepsilon_A\tensor\bangterminalmap_A\tensor\bangcoDiag_A}
&&
A\tensor\I\tensor\Fock A \iso A\tensor\Fock A
})
\\
+ &
\\
& (\xymatrix@C55pt{
A\tensor\Fock A\iso A\tensor\I\tensor\Fock A
\ar[r]^-{\eta_A\tensor\banginitialmap_A\tensor\bangDiag_A}&
\Fock A\tensor\Fock A\tensor\Fock A\tensor\Fock A}
\\
&
\qquad
\xymatrix@C25pt{
\ar[rrr]^-{\id_{\Fock A}\tensor\symm_{\Fock A,\Fock A}\tensor\id_{\Fock A}}
&&&
\Fock A\tensor\Fock A\tensor\Fock A\tensor\Fock A
\ar[rr]^-{
\bangterminalmap_A
\tensor
\varepsilon_A
\tensor
\bangcoDiag_A
}
&&
\I 
\tensor
A\tensor\Fock A
\iso A\tensor\Fock A
})
\\
+ &
\\
& (\xymatrix@C55pt{
A\tensor\Fock A\iso 
\I \tensor A\tensor \Fock A 
\ar[r]^-{ \banginitialmap_A \tensor \eta_A \tensor \bangDiag_A }&
\Fock A\tensor\Fock A\tensor\Fock A\tensor\Fock A}
\\
&
\qquad
\xymatrix@C25pt{
\ar[rrr]^-{\id_{\Fock A}\tensor\symm_{\Fock A,\Fock A}\tensor\id_{\Fock A}}
&&&
\Fock A\tensor\Fock A\tensor\Fock A\tensor\Fock A
\ar[rr]^-{
\bangterminalmap_A 
\tensor 
\varepsilon_A
\tensor 
\bangcoDiag_A 
}
&&
\I 
\tensor 
A\tensor \Fock A 
\iso A\tensor\Fock A
})
\end{array}
$$
which, in turn, by the bialgebra laws and
Lemma~\ref{TechnicalLemma}~$(\ref{TechnicalLemmaZero})$, equals
$$
\big((\varepsilon_A\icomp\eta_A)\tensor\id_{\Fock A}\big)
+ 
0_{A\tensor\Fock A,A\tensor\Fock A}
+ 
0_{A\tensor\Fock A,A\tensor\Fock A}
+
\big(
(\id_A\tensor\creation\eta_A)
(\symm_{A,A}\tensor\id_{\Fock A})
(\id_A\tensor\annihilation\varepsilon_A)
\big)
$$

$(\ref{Commutation_Relations_TWO})~\&~(\ref{Commutation_Relations_THREE})$ The
arguments crucially rely on the commutativity of the Fock-space bialgebra
structure.  Since the two arguments are dual of each other, I only consider
one of them.  
$$\xymatrix{
  A\tensor A\tensor\Fock A 
  \ar[d]_-{\symm_{A,A}\tensor\id_{\Fock A}}
  \ar[r]^-{\id_A\tensor\eta_A\tensor\id_{\Fock A}} 
& A\tensor \Fock A\tensor\Fock A 
  \ar[d]_-{\symm_{A,\Fock A}\tensor\id_{\Fock A}}
  \ar[rr]^-{\id_{\Fock A}\tensor\bangcoDiag_A} 
  \ar[dr]^-{\eta_A\tensor\id_{\Fock A}\tensor\id_{\Fock A}}
& 
& A\tensor\Fock A 
\ar[dr]^-{\eta_A\tensor\id_{\Fock A}} 
&
\\
  A\tensor A\tensor \Fock A 
  \ar[r]_-{\eta_A\tensor\id_A\tensor\id_{\Fock A}}
  \ar[d]_-{\id_A\tensor\eta_A\id_{\Fock A}}
& \Fock A\tensor A\tensor\Fock A 
  \ar[dr]_-{\id_{\Fock A}\tensor\eta_A\tensor\id_{\Fock A}}
& \Fock A\tensor\Fock A\tensor\Fock A
  \ar[d]|-{\symm_{\Fock A,\Fock A}\tensor\id_{\Fock A}}
  \ar[rr]_-{\id_{\Fock A}\tensor\bangcoDiag_A}
  \ar[dr]^-{\bangcoDiag_A\tensor\id_{\Fock A}} &
& \Fock A\tensor\Fock A 
  \ar[d]^-{\bangcoDiag_A} 
\\
  A\tensor\Fock A\tensor\Fock A 
  \ar[rr]_-{\eta_A\tensor\id_{\Fock A}\tensor\id_{\Fock A}}
  \ar[dr]_-{\id_A\tensor\bangcoDiag_A}
& 
& \Fock A\tensor\Fock A\tensor\Fock A
\ar[dr]_-{\id_{\Fock A}\tensor\bangcoDiag_A}
  \ar[r]_-{\bangcoDiag_A\tensor\id_{\Fock A}} 
& \Fock A\tensor\Fock A
  \ar[r]_-{\bangcoDiag_A} 
& \Fock A 
\\
& A\tensor\Fock A
  \ar[rr]_-{\eta_A\tensor\id_{\Fock A}}
& 
& \Fock A\tensor\Fock A
  \ar[ur]_-{\bangcoDiag_A}
& 
}$$\\[-12.5mm]
\end{proof}

Analogously, one can establish the following laws of interaction between
the creation/annihilation operators and the bialgebra structure.
\begin{proposition}
For a Fock-space construction $\Fock$, the following hold for all natural
transformations $\eta_A:A\rightarrow\Fock A$ and $\varepsilon_A:\Fock
A\rightarrow A$.
\begin{enumerate}
\item\label{TechnicalLemmaThree}
$\bangterminalmap_A\icomp\creation\eta_A = 0_{A\tensor\Fock A,\I}$
and
$\banginitialmap_A\icomp\annihilation\varepsilon_A 
 = 0_{\I,A\tensor\Fock A}$.

\item\label{TechnicalLemmaFour}
  $\bangDiag_A\icomp\creation\eta_A 
  = \big((\creation\eta_A\tensor\id_{\Fock A})+(\id_{\Fock
    A}\tensor\creation\eta_A)\icomp(\sigma_{A,\Fock A}\tensor\id_{\Fock
    A})\big)\icomp(\id_A\tensor\bangDiag_A) 
    : A\tensor\Fock A\rightarrow\Fock A\tensor\Fock A$
  and 
  $\annihilation\varepsilon_A\icomp\bangcoDiag_A 
   = (\id_A\tensor\bangcoDiag_A)
     \icomp
     \big((\annihilation\varepsilon_A\tensor\id_{\Fock A})
      +
      (\symm_{\Fock A,A}\tensor\id_{\Fock A})(\id_{\Fock
      A}\tensor\annihilation\varepsilon_A)\big)
   : \Fock A\tensor\Fock A\rightarrow A\tensor\Fock A$.
\end{enumerate}
\end{proposition}
\hide{%%% BEGIN HIDE !!!
\begin{proof}
$(\ref{TechnicalLemmaThree})$ \ldots

$(\ref{TechnicalLemmaFour})$
The core of the argument is given by the commutativity of the diagram
below
$$
\hspace{-5mm}
\xymatrix@C35pt{
A\tensor\Fock A 
\ar[d]|-{\id_A\tensor\bangDiag_A}
\ar[r]^-{\eta_A\tensor\id_{\Fock A}} & \Fock A\tensor\Fock A
\ar[rr]^-{\bangcoDiag_A} 
\ar[d]|-{\bangDiag_A\tensor\bangDiag_A} &
& \Fock A\ar[ddd]|-{\bangDiag_A} & 
\\
A\tensor\Fock A\tensor\Fock A
\ar[dd]|-{\eta_A\tensor\id_{\Fock A}\tensor\id_{\Fock A}}
&
\Fock A\tensor\Fock A\tensor\Fock A\tensor\Fock A
\ar[dr]|-{\id_{\Fock A}\tensor\symm_{\Fock A,\Fock A}\tensor\id_{\Fock A}} 
& 
&
\\
& 
&
\Fock A\tensor\Fock A\tensor\Fock A\tensor\Fock A
\ar[rd]_-{\bangcoDiag_A\tensor\bangcoDiag_A}
& 
& 
\\
\Fock A\tensor\Fock A\tensor\Fock A
\ar[ruu]|(.6){\enspace(\q1+\q2)\tensor\id_{\Fock A}\tensor\id_{\Fock A}}
\ar[rru]|(.55){\quad\qquad(\id_{\Fock A}\tensor\id_{\Fock
A}\tensor\q2)+((\q2\tensor\id_{\Fock A}\tensor\id_{\Fock
A})(\symm_{\Fock A,\Fock A}\tensor\id_{\Fock A}))}
\ar[rrr]_-{(\bangcoDiag_A\tensor\id_{\Fock A})+((\id_{\Fock
A}\tensor\bangcoDiag_A)(\symm_{\Fock A,\Fock A}\tensor\id_{\Fock A}))} &&&
\Fock A\tensor\Fock A 
}$$
where one sees the crucial role played by the bialgebra structure of Fock
space and item~$(\ref{TechnicalLemmaTwo})$ of Lemma~\ref{TechnicalLemma}.
\end{proof}
}%%% END HIDE !!!

\subsection{Coherent states}
\label{CoherentStates}

Our discussion of coherent states is within the framework of categorical
models of linear logic, see~\eg~\cite{Mellies}.

\begin{definition}
A \emph{linear Fock-space construction} is one equipped with \emph{linear
exponential comonad structure}~$(\epsilon,\delta)$ in the form of natural
transformations $\epsilon_A:\Fock A\rightarrow A$ and $\delta_A:\Fock
A\rightarrow\Fock\Fock A$ such that 
$$
\xymatrix{
& \ar[d]|-{\delta_A} \ar[dl]_-{\id_{\Fock A}} \Fock A \ar[rd]^-{\id_{\Fock
A}} & 
\\
\Fock A & \ar[l]^-{\epsilon_{FA}} \Fock \Fock A \ar[r]_-{\Fock\epsilon_A}
& \Fock A }
\qquad\qquad
\xymatrix{
\Fock A \ar[d]_-{\delta_A} \ar[r]^-{\delta_A} & \Fock \Fock A
\ar[d]^-{\delta_{\Fock A}}
\\
\Fock \Fock A \ar[r]^-{\Fock \delta_A} & \Fock \Fock A 
} 
$$
and subject to the coherence conditions 
$$
\xymatrix{
& \Fock\Fock\O \ar[dr]^-{F\terminalmap_{\Fock\O}} & 
\\
\ar[ru]^-{\delta_\O} 
\Fock\O \ar[rr]_-{\id_{\Fock\O}} && \Fock\O
}
\qquad
\xymatrix@C45pt{
\Fock A\tensor\Fock B \ar[d]_-{\varphi_{A,B}}
\ar[rr]^-{\delta_A\tensor\delta_B} & & \Fock \Fock A\tensor\Fock \Fock B 
\ar[d]^-{\varphi_{\Fock A,\Fock B}}
\\
\Fock (A\biprod B) \ar[r]_-{\delta_{A\biprod B}} & \Fock \Fock (A\biprod
B) \ar[r]_-{\Fock\pair{\Fock\proj1,\Fock\proj2}} & 
\Fock(\Fock A\biprod\Fock B)
}$$
\end{definition}

\begin{definition}\label{CoherentStateDefinition}
Let $\Fock$ be a linear Fock-space construction.  A \emph{coherent state}
$\gamma$ is a map $\I\rightarrow \Fock A$ such that 
\begin{enumerate}
  \item \label{CoherentStateDefinitionOne}
    $\annihilation\epsilon_A \icomp \gamma 
    = (\xymatrix{\I\iso\I\tensor\I \ar[r]^-{v\tensor\gamma} &
       A\tensor\Fock A})$
    for some $v:\I\rightarrow A$,
  \item \label{CoherentStateDefinitionTwo}
    $\bangterminalmap_A\icomp\gamma=\id_\I$, and 
  \item \label{CoherentStateDefinitionThree}
    $\bangDiag_A\icomp\gamma
    = (\xymatrix{\I\iso\I\tensor\I\ar[r]^-{\gamma\tensor\gamma}&\Fock
    A\tensor\Fock A})$.
\end{enumerate}
\end{definition}

\begin{definition}\label{ExtensionsDef}
Let $\Fock$ be a linear Fock-space construction.
\begin{enumerate}
\item
The \emph{Kleisli extension} $\Kext u:\Fock X\rightarrow\Fock A$ of
$u:\Fock X\rightarrow A$ is defined as $\Fock(u)\comp\delta_X$.

\item\label{GlobalElementExtension}
The \emph{extension} $\ext v:\I\rightarrow\Fock A$ of
$v:\I\rightarrow A$ is the composite
$$\xymatrix{
\I \iso\Fock\O\ar[r]^-{\delta_\O} &\Fock\Fock O \iso\Fock\I
\ar[r]^-{\Fock v} & \Fock A }
%\enspace.
$$
\end{enumerate}
\end{definition}
For instance,
$\ext{0_{\I,A}} = \banginitialmap_A:\I\rightarrow\Fock A$.

\begin{theorem}\label{CoherentStateTheorem}
For every $v:\I\rightarrow A$, the extension $\ext v:\I\rightarrow\Fock A$
is a coherent state.
\end{theorem}

The theorem arises from the following facts.

\begin{proposition}
Let $\Fock$ be a linear Fock-space construction.
\begin{enumerate}
  \item 
    For $f:A\rightarrow B$, 
    $\bangterminalmap_B\comp\Fock(f)
     =\bangterminalmap_A:\Fock A\rightarrow\I$.

  \item 
    $\bangterminalmap_{\Fock A}\icomp\delta_A=\bangterminalmap_A:
     \Fock A\rightarrow\I$.

  \item
    $\bangDiag_{\Fock A}\icomp\delta_A
     = (\delta_A\tensor\delta_A)\icomp\bangDiag_A
     : \Fock A\rightarrow\Fock\Fock A\tensor\Fock\Fock A$.

  \item
    For $u:\Fock X\rightarrow A$, 
    $\annihilation\epsilon_A\comp\Kext u 
    = (u\tensor\Kext u)\icomp\bangDiag_X$.

  \item
    $\bangDiag_\O =
    (\Fock\O\iso\I\iso\I\tensor\I\iso\Fock\O\tensor\Fock\O)$.
\end{enumerate}
\end{proposition}

We conclude the section by recording a property that will be useful at the
end of the paper.

\begin{proposition}\label{CoherentStateCreation}
Let $\eta_A:A\rightarrow\Fock A$ be a natural transformation for a linear
Fock-space construction~$\Fock$. For $v:\I\rightarrow A$,
\begin{equation}\label{CoherentStateCreationEquation}
\hfill
\Kext{(\creation\eta_A^v)}\icomp\banginitialmap_A
= \,
\ext{\eta_A\icomp v}
%\enspace.
\hfill\hfill
\end{equation}
\end{proposition}

\section{Combinatorial model}
\label{CombinatorialModel}

I introduce and study a model for Fock space with creation/annihilation
operators that arises in the setting of \emph{generalised species of
structures}~\cite{FioreFOSSACS,FGHW}.  These are a categorical
generalisation of both the structural combinatorial theory of species of
structures~\cite{JoyalAdv,Joyal1234,BLL} and the relational model of
linear logic. %~\cite{}.

Our combinatorial model conforms to the axiomatics of the previous section by
being an example of its generalisation from categories to
\emph{bicategories}~\cite{BenabouBiCat}, by which I~roughly mean the
categorical setting where all structural identities hold up to canonical
coherent isomorphism.  However, we will not dwell on this here.

\subsection{The bicategory of profunctors}
\label{TheBicategoryOfProfunctors}

Our setting for spaces and linear maps will be the \emph{bicategory of
profunctors}~$\Prof$, for which see~\eg~\cite{LawvereGenMet,BenabouDist}.  A
\emph{profunctor} (or \emph{bimodule}, or \emph{distributor}) $\scat
A\profrightarrow\scat B$ between small categories $\scat A$ and $\scat B$
is a functor $\scat A^\op\times\scat B\rightarrow \Set$.  It might be
useful to think of these as category-indexed set-valued matrices.

The bicategory $\Prof$ has objects given by small categories, maps given
by profunctors, and $2$-cells given by natural transformations.  The
profunctor composition $T S:\scat A\profrightarrow\scat C$ of
$S:\scat A\profrightarrow\scat B$ and $T:\scat B\profrightarrow\scat C$ is
given by the matrix-multiplication formula
\begin{equation}\label{CoendCompositionFormula}
\textstyle
\hfill
T S\,(a,c) = \coend^{b\in\scat B} S(a,b)\times T(b,c)
\hfill\hfill
\end{equation}
where $\times$ and $\coend$ respectively denote the cartesian product and
coend operations.  The associated identity profunctors $I_{\scat C}$ are
the hom-set functors $\scat C^\op\times\scat
C\rightarrow\Set:(c',c)\mapsto\scat C(c',c)$.  

The notion of coend and its properties, for which
see~\eg~\cite[Chapter~X]{MacLane}, is central to the calculus of this
section.  A coend is a colimit arising as a coproduct under a quotient
that establishes compatibility under left and right actions.  Technically,
the coend $\coend^{z\in\scat C} H(z,z)\in\Set$ of a functor $H:\scat
C^\op\times\scat C\rightarrow\Set$ can be presented as the following
coequaliser: 
$$\xymatrix@R=7.5pt{
(f:x\rightarrow y,h) \ar@{|->}[r] & \big(x,H(f,\id_x)(h)\big)
& 
\\
\coprod_{f:x\rightarrow y\, \mathrm{in}\, \scat C} H(y,x) \ar@<.5em>[r]
\ar@<-.5em>[r] & \coprod_{z\in\scat C} H(z,z)
\ar@{->>}[r]
&
\coend^{z\in\scat C} H(z,z) \\
(f:x\rightarrow y,h) \ar@{|->}[r] & \big(y,H(\id_y,f)(h)\big)
&
}$$
As for $(\ref{CoendCompositionFormula})$, then, $TS\,(a,c)$ consists of
equivalence classes of triples in $\coprod_{b\in\scat B} S(a,b)\times
T(b,c)$ under the equivalence relation generated by identifying
$\big(b,s,T(f,\id_{b'})(t')\big)$ and $\big(b',S(\id_a,f)(s),t'\big)$ for
all $f:b\rightarrow b'$ in $\scat B$, $s\in S(a,b)$, $t'\in T(b',c)$.
Note also that, for all $P:\scat C^\op\rightarrow\Set$, there is a
canonical natural isomorphism 
\begin{equation}\label{YonedaLemmaIso}\textstyle
\hfill
P(c)
\iso
\coend^{z\in\scat C} P(z) \times \scat C(c,z) 
\enspace,
\hfill\hfill
\end{equation}
known as the \emph{density formula}~\cite{MacLane} or \emph{Yoneda
lemma}~\cite{KellyEnrichedCT}, that essentially embodies the unit laws of
profunctor composition with the identities.

The bicategory $\Prof$ not only has compatible biproduct and
symmetric monoidal structures but is in fact a \emph{compact closed
bicategory}, see~\cite{DayStreet}.  The biproduct structure is given by
the empty and binary coproduct of categories (\ie~$\O=\Incat$ and
$\biprod=+$), and the tensor product structure is given by the empty and
binary product of categories (\ie~$\I=\Tercat$ and $\tensor=\times$).

\begin{remark}
The analogy of profunctors between categories as matrices between bases can be
also phrased as an analogy between cocontinuous functors between presheaf
categories and linear transformations between free vector spaces.  

As it is well-known, %~\cite{}, 
the free small-colimit completion of a small category~$\scat C$ is the
functor category $\Set^{\scat C^\op}$ of (contravariant) presheaves on
$\scat C$ and natural transformations between them.  The universal map is
the \emph{Yoneda embedding} $\scat C \xymatrix@C10pt{\ar@{^(->}[r]&}
\Set^{\scat C^\op}: z\mapsto\,\ket z$ where 
$$
\ket z :\scat C^\op \rightarrow \Set : c \mapsto \scat C(c,z)
%\enspace.
$$
The use of Dirac's ket notation in this context is justified by regarding
presheaves as vectors and noticing that the
isomorphism~$(\ref{YonedaLemmaIso})$ above amounts to the following one
$$\textstyle
P \,\iso\, \coend^{z\in\scat C} P_z\ \cdot \ket z
$$ in $\Set^{\scat C^\op}$ expressing every presheaf as a colimit of the
basis vectors (referred to as \emph{representable presheaves} in
categorical terminology).  Associated to this representation, the notion
of linearity for transformations corresponds to that of cocontinuity
(\ie~colimit preservation) for functors.
Indeed, the bicategory of profunctors is biequivalent to the $2$-category with
objects consisting of small categories, morphisms from $\scat A$ to $\scat B$
given by cocontinuous functors $\Set^{\scat A}\rightarrow\Set^{\scat B}$, and
\mbox{$2$-cells} given by natural transformations. %; see~\eg~\cite{}.  
The biequivalence associates a profunctor $T:\scat A\profrightarrow\scat B$
with the cocontinuous functor $\Fun(T):\Set^{\scat A}\rightarrow\Set^{\scat B}
: P \mapsto \coend^{b\in\scat B} \big[ \coend^{a\in\scat A} P_a \times
T(a,b)\big]\;\cdot \ket b$, whilst the profunctor $\Pro(F):\scat
A\profrightarrow\scat B$ underlying a cocontinous functor $F:\Set^{\scat
A}\rightarrow\Set^{\scat B}$ has entry $F\!\ket a_b$ at $(a,b)\in\scat
A^\op\times\scat B$.  In particular, note the following: 
$$\begin{array}{l}
\Fun(\Pro\, F)(P) 
\\[1mm]
\quad = 
\coend^{b\in\scat B}
    \big[ \coend^{a\in\scat A} P_a \times F\!\ket a_b\big]\, \cdot \ket b
\iso
\coend^{a\in\scat A} 
  P_a \cdot \big(\coend^{b\in\scat B} F\!\ket a_b\; \cdot \ket b\big)
\\[1.5mm]
\quad\iso
\coend^{a\in\scat A} P_a \cdot F\!\ket a
\iso
F\big(\coend^{a\in\scat A} P_a \cdot \ket a\big)
\mbox{ , by cocontinuity}
\\[.5mm]
\quad 
\iso F(P)
\\[3mm]
\Pro(\Fun\, T)(a,b) 
\\ \quad
=
\big( \coend^{y\in\scat B} \big[ \coend^{x\in\scat A} \ket a_x \times
T(x,y)\big] \cdot \ket y\big)_b 
\iso
\big( \coend^{y\in\scat B} T(a,y) \cdot \ket y\big)_b 
\iso
T(a,b)
\end{array}$$
\end{remark}

\subsection{Combinatorial Fock space}
\label{CombinatorialFockSpace}

Let us introduce the combinatorial Fock-space construction.

\begin{definition}
The \emph{combinatorial Fock space} of a small category $\scat C$ is the
small category 
$$\textstyle
\freesmc\, \scat C 
= 
\coprod_{n\in\Nat} \scat C^n\quotient{\symmgroup_n}
$$
where $\scat C^n\quotient{\symmgroup_n}$ has objects given by $n$-tuples
of objects of $\scat C$ and hom-sets
$$\textstyle
\scat C^n\quotient{\symmgroup_n}(\vec c,\vec z)
= \coprod_{\sigma\in\symmgroup_n} 
    \prod_{1\leq i\leq n}
      \scat C(c_i,z_{\sigma i})
%\enspace.
$$
\end{definition}
It is a very important part of the general theory, for which
see~\cite{FioreFOSSACS,FGHW}, that the combinatorial Fock-space construction
is the \emph{free symmetric (strict) monoidal completion}; the unit and tensor
product being respectively given by the empty tuple and tuple concatenation,
and denoted as~$(\,)$ and $\concat$\,.

\begin{proposition}\label{HomsetCombinatorics}
Hom-sets in combinatorial Fock space satisfy the following combinatorial
laws.
\begin{enumerate}

\item\label{CombinatorialBialgebraLaw}
$\begin{array}[t]{l}
   \freesmc\scat A(\vec u\concat\vec v,\vec x\concat\vec y)
\\
\enspace \iso \,
   \coend^{\vec a,\vec b,\vec c,\vec d\in\freesmc\scat A}
     \
     \freesmc\scat A(\vec u,\vec a\concat\vec b)
     \times
     \freesmc\scat A(\vec v,\vec c\concat\vec d)
     \times
     \freesmc\scat A(\vec a\concat\vec c,\vec x)
     \times
     \freesmc\scat A(\vec b\concat\vec d,\vec y)
\end{array}$

\item
$\begin{array}[t]{ll}
\freesmc\scat A\big( (\,) , (\,)\big)
\iso
1
&,\enspace
\freesmc\scat A\big( (a) , (x)\big)
\iso
\scat A( a , x)
\\[2mm]
\freesmc\scat A\big( (\,) , (a)\big)
\iso
0
&,\enspace
\freesmc\scat A\big( (a),(\,) \big)
\iso
0
\end{array}$

\item
$\begin{array}[t]{ll}
\freesmc\scat A\big( (\,) , \vec x\concat\vec y \big)
\iso
\freesmc\scat A\big( (\,) , \vec x \big)
\times
\freesmc\scat A\big( (\,) , \vec y \big)
& ,\enspace
\\[2mm]
\freesmc\scat A\big( \vec x\concat\vec y , (\,) \big)
\iso
\freesmc\scat A\big( \vec x, (\,) \big)
\times
\freesmc\scat A\big( \vec y , (\,) \big)
\end{array}$

\item \label{HomsetCombinatoricsFour}
$\begin{array}[t]{l}
\freesmc\scat A\big( (a),\vec x\concat\vec y\big)
 \iso
 \big(\, \freesmc\scat A\big( (a),\vec x\big) 
   \times \freesmc\scat A\big( (\,),\vec y\big) 
   \,\big)
 \,+\,
 \big(\, \freesmc\scat A\big( (\,),\vec x\big) 
   \times \freesmc\scat A\big( (a),\vec y \big) 
   \,\big)
\\[2mm]
\freesmc\scat A\big(\vec x\concat\vec y,(a)\big)
 \iso
 \big(\, \freesmc\scat A\big(\vec x,(a)\big) 
   \times \freesmc\scat A\big(\vec y,(\,)\big) 
   \,\big)
 \, + \,
 \big(\, \freesmc\scat A\big(\vec x,(\,)\big) 
   \times \freesmc\scat A\big(\vec y,(a)\big) 
   \,\big)

\end{array}$

\item
$\begin{array}[t]{c}
\freesmc(\scat A+\scat B)\,
  \big(\freesmc\inj1(\vec a)\concat\freesmc\inj2(\vec b),
   \freesmc\inj1(\vec x)\concat\freesmc\inj2(\vec y)\big)
\ \iso \
\freesmc\scat A(\vec a,\vec x)
\times
\freesmc\scat B(\vec b,\vec y)
\end{array}$
\end{enumerate}
\end{proposition}

\medskip
I proceed to describe the structure of the combinatorial Fock space.

\paragraph{$\boldsymbol \S$\enspace 2.2.1.}
For a profunctor $T:\scat A\profrightarrow\scat B$, the profunctor
$\freesmc\,T:\freesmc\,\scat A\profrightarrow\freesmc\scat B$ is given by 
$$\textstyle
\freesmc T\,(\vec x,\vec y)
\,=\,
\coend^{\vec z\in\freesmc(\scat A^\op\times\scat B)}
  \big(\prod_{z_i\in\vec z} T\,z_i\big)
  \times
  \freesmc\scat A(\vec x,\freesmc\proj1\vec z)
  \times
  \freesmc\scat B(\freesmc\proj2\vec z,\vec y)
$$
so that
$$\begin{array}{l}
\freesmc\, T\, \big( (a_1,\ldots,a_m),(b_1,\ldots,b_n)\big)
\iso \left\{\begin{array}{ll}
      \coprod_{\sigma\in\symmgroup_m} 
        \prod_{1\leq i\leq m} T(a_i,b_{\sigma i}) 
        & \mbox{, if $m=n$}
      \\[1mm]
      0 & \mbox{, otherwise}
      \end{array}\right.
    \end{array}$$

\paragraph{$\boldsymbol \S$\enspace 2.2.2.}
There are canonical natural coherent equivalences as follows:
$$\begin{array}{ll}
    \phi: \Tercat \twoequiv \freesmc\,\Incat 
     &,\quad
    \phi\big(\, * , (\,) \,\big) = 1
    \\[2mm]
    \varphi_{\scat A,\scat B}
    : \freesmc \scat A\times\freesmc\scat B 
      \twoequiv
      \freesmc(\scat A+\scat B)
     &,\quad
    \varphi_{\scat A,\scat B}\big(\,(\vec x,\vec y),\vec z\,) 
    = \freesmc(\scat A+\scat B)
        (\freesmc\inj1(\vec x)\concat\freesmc\inj2(\vec y),\vec z)
    \end{array}$$

\paragraph{$\boldsymbol \S$\enspace 2.2.3.}
The pseudo commutative bialgebra
structure~(\ref{FockSpaceCommutativeBialgebraData}) consists of:
    $$\begin{array}{ll}
    \banginitialmap_{\scat A}:\Tercat\profrightarrow\freesmc\scat A
     &,\quad
     \banginitialmap_{\scat A}\big( *,\vec a \big) 
     = \freesmc\scat A\big( (\,) , \vec a \big)
    \\[3mm]
      \bangcoDiag_{\scat A}:\freesmc\scat A\times\freesmc\scat
      A\profrightarrow\freesmc\scat A
     &,\quad
      \bangcoDiag_{\scat A}\big( (\vec x,\vec y) , \vec z \big)
      = \freesmc\scat A(\vec x\concat \vec y , \vec z )
    \\[3mm]
    \bangterminalmap_{\scat A}: \freesmc\scat A \profrightarrow \Tercat
     &,\quad
     \bangterminalmap_{\scat A}\big( \vec a,* \big) 
     = \freesmc\scat A\big( \vec a, (\,) \big)
    \\[3mm]
      \bangDiag_{\scat A}: \freesmc\scat A \profrightarrow
      \freesmc\scat A\times\freesmc\scat A
     &,\quad
      \bangDiag_{\scat A}\big( \vec z , (\vec x,\vec y) \big)
      = \freesmc\scat A(\vec z , \vec x\concat \vec y )
   \end{array}$$

   The bialgebra law for $\bangcoDiag_{\scat A}\icomp \bangDiag_{\scat A}$
   arises from the combinatorial law of
   Proposition~\ref{HomsetCombinatorics}~$(\ref{CombinatorialBialgebraLaw})$,
   which is a formal expression for the diagrammatic law:\\
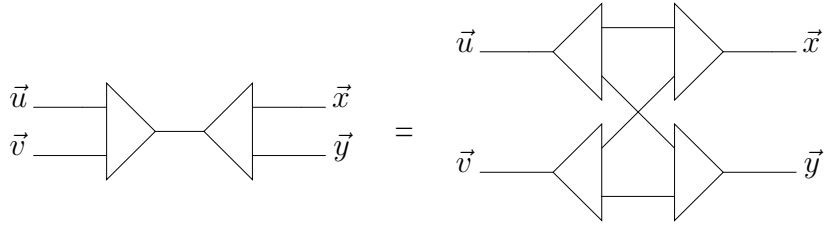
\begin{center}
\begin{tabular}{c}
\begin{minipage}{5cm}
\input{bialg1.eepic}
\end{minipage}
=
\quad
\begin{minipage}{5cm}
\input{bialg2.eepic}
\end{minipage}
\end{tabular}
\end{center}

\paragraph{$\boldsymbol \S$\enspace 2.2.4.}
The linear exponential pseudo comonad structure is given by:
$$\begin{array}{ll}
     \epsilon_{\scat A}:\freesmc\scat A\profrightarrow\scat A
     &,\quad
     \epsilon_{\scat A}(\vec x,a) = \freesmc\scat A\big( \vec x , (a) \big)
     \\[3mm]
     \delta_{\scat A}:\freesmc\scat A\profrightarrow\freesmc\freesmc\scat A
     &,\quad
     \delta_{\scat A}(\vec a, \alpha)
     = 
     \freesmc\scat A(\vec a,\flatten{\alpha})
   \end{array}$$
   where $\flatten{(\vec a_1,\ldots,\vec a_n)} 
   = \vec a_1\cdot\ldots\cdot\vec a_n\in\freesmc\scat A$
   for $\vec a_i\in\freesmc\scat A$.

   The laws of
   Proposition~\ref{HomsetCombinatorics}~$(\ref{HomsetCombinatoricsFour})$
   exhibit the combinatorial context of the identities of
   Proposition~\ref{TechnicalLemma}~$(\ref{TechnicalLemmaTwo})$.

\paragraph{$\boldsymbol \S$\enspace 2.2.5.}
\label{ProfDualityAndMonad}
The bicategory $\Prof$ admits a \emph{duality}, by which a small
category $\scat A$ is mapped to its opposite category $\scat A^\op$ and a
profunctor $T:\scat A\profrightarrow\scat B$ to the profunctor
$T^\op:\scat B^\op\profrightarrow\scat A^\op$ with $T^\op(\vec y,\vec x) =
T(\vec x,\vec y)$.  Thereby, the pseudo \emph{comonadic} structure of the
combinatorial Fock-space construction can be turned into pseudo
\emph{monadic} structure~$(\eta,\mu)$ by setting $\eta_{\scat A} =
(\epsilon_{\scat A^\op})^\op$ and $\mu_{\scat A} = (\delta_{\scat
A^\op})^\op$.  Specifically, we have:
$$\begin{array}{ll}
     \eta_{\scat A}:\scat A\profrightarrow\freesmc\scat A
     &,\quad
     \eta_{\scat A}(a,\vec x) = \freesmc\scat A\big( (a),\vec x \big)
     \\[3mm]
     \mu_{\scat A}:\freesmc\freesmc\scat A\profrightarrow\freesmc\scat A
     &,\quad
     \mu_{\scat A}(\alpha,\vec a\big)
     = 
     \freesmc\scat A(\flatten\alpha,\vec a)
   \end{array}$$

\paragraph{$\boldsymbol \S$\enspace 2.2.6.}
The structure results in canonical creation and annihilation operators:
$$\begin{array}{ll}
\creation\eta_{\scat A}
:\scat A\times\freesmc\scat A\profrightarrow\freesmc\scat A
&,\quad
\creation\eta_{\scat A}\big( (a,\vec x), \vec y)\big)
= 
\freesmc\scat A( \vec x\concat(a) , \vec y )
\\[3mm]
\annihilation\epsilon_{\scat A}
: \freesmc\scat A\profrightarrow\scat A\times\freesmc\scat A
&,\quad
\annihilation\epsilon_{\scat A}\big( \vec x , (a,\vec y) \big)
=
\freesmc\scat A( \vec x , (a)\concat\vec y )
\end{array}$$
so that, for $V:\Tercat\profrightarrow\scat A$ and $V':\scat
A\profrightarrow\Tercat$, we have\\[-4mm]
\begin{eqnarray}
\label{CreationCoendFormula}
&
\textstyle
\creation\eta_{\scat A}^V(\vec x,\vec y) 
\,\iso\,
\coend^{a\in\scat A}
V_a\times\freesmc\scat A( \vec x\concat(a) , \vec y )
& \quad\mbox{ for $V_a=V(*,a)$}
\\[3mm] \nonumber
&\textstyle
\annihilation\epsilon_{\scat A}^{V'}(\vec x,\vec y)
\,\iso\,
\coend^{a\in\scat A}
V'_a\times\freesmc\scat A( \vec x , (a)\concat\vec y )
&\quad\mbox{ for $V'_a=V(a,*)$}
\end{eqnarray}
yielding the functorial forms
$$\begin{array}{l}
(\Fun\,\creation\eta_{\scat A}^V)(X)
\,\iso\,
\coend^{a\in\scat A,\vec z\in\freesmc\scat A}
  \,\big[\,V_a\times X_{\vec z}\,\big]\,\, \cdot \ket{\vec z\cdot(a)}
\\[4mm]
(\Fun\,\annihilation\epsilon^{V'}_{\scat A})(X)
\,\iso\,
\coend^{a\in\scat A,\vec z\in\freesmc\scat A}
  \,\big[\,V'_a\times X_{(a)\cdot\vec z}\,\big]\,\, \cdot \ket{\vec z}
\end{array}$$

Identity~(\ref{VectorCovectorCreationAnnihilation}) then becomes
$$\textstyle
\Fun(\annihilation\epsilon^{V'}_{\scat A}\icomp\creation\eta^V_{\scat A})
\,(X)
\ \iso\
\langle V,V'\rangle \cdot X
+
\coend^{a,b\in\scat A,\vec z\in\freesmc\scat A}
  \big[\, V_a \times V'_b
  \times 
  X_{(b)\cdot\vec z}\,\big]\ \cdot \ket{\vec z \cdot a}
$$
where $\langle V,V'\rangle = \coend^{a\in\scat A} V_a\times V'_a$.

\paragraph{$\boldsymbol \S$\enspace 2.2.7.}
In the current setting, the axiomatic proof of the commutation relation
for $\annihilation\epsilon_{\scat A}\icomp\creation\eta_{\scat A}$
acquires formal combinatorial content made explicit by the following chain
of isomorphisms:
\begin{eqnarray}
\nonumber
&&
\hspace{-22.5mm}
\annihilation\epsilon_{\scat A}\icomp\creation\eta_{\scat A}
  \big( (a,\vec x) , (b,\vec y) \big)
\\[2mm]
&&
\label{CombinatorialContentOne}
\hspace{-22.5mm}
\quad \iso\
  \freesmc\scat A(\vec x\concat (a) , (b) \concat \vec y)
\\[2mm]
\nonumber
&&\textstyle
\hspace{-22.5mm}
\quad \iso\
\coend^{\vec{z_1},\vec{z_2},\vec{z_3},\vec{z_4}\in\freesmc\scat A}
  \ 
  \freesmc\scat A(\vec x, \vec{z_1} \concat \vec{z_2} )
  \times
  \freesmc\scat A\big( (a), \vec{z_3} \concat \vec{z_4} \big)
  \times
  \freesmc\scat A\big(\vec{z_1}\concat\vec{z_3}, (b) \big)
  \times
  \freesmc\scat A(\vec{z_2}\concat\vec{z_4}, \vec y)
\\[2mm]
&&\nonumber\textstyle
\hspace{-22.5mm}
\quad \iso\
\coend^{\vec{z_1},\vec{z_2},\vec{z_3},\vec{z_4}\in\freesmc\scat A}
  \begin{array}[t]{l}
  \freesmc\scat A(\vec x, \vec{z_1} \concat \vec{z_2} )
  \\
  \quad \times\
  \big[
  \freesmc\scat A\big( (a), \vec{z_3} \big)
  \times
  \freesmc\scat A\big( (\,), \vec{z_4} \big)
  +
  \freesmc\scat A\big( (\,), \vec{z_3} \big)
  \times
  \freesmc\scat A\big( (a), \vec{z_4} \big)
  \big]
  \\
  \quad \times\
  \big[
  \freesmc\scat A\big(\vec{z_1}, (b) \big)
  \times
  \freesmc\scat A\big(\vec{z_3}, (\,) \big)
  +
  \freesmc\scat A\big(\vec{z_1}, (\,) \big)
  \times
  \freesmc\scat A\big(\vec{z_3}, (b) \big)
  \big]
  \\
  \quad \times\
  \freesmc\scat A(\vec{z_2}\concat\vec{z_4}, \vec y)
  \end{array}
\\[1mm]
\nonumber
&&
\hspace{-22.5mm}
%$$\begin{eqnarray}
\quad \iso\
  \big[ \freesmc\scat A(\vec x, (b)\concat \vec y )
  \times
  \freesmc\scat A\big( (a), (\,) \big)
  \big]
\,+\,
  \big[ \freesmc\scat A(\vec x, \vec{y} )
  \times
  \freesmc\scat A\big( (a), (b) \big)
  \big]
\\[1mm]
\nonumber
&&
\hspace{-22.5mm}
\textstyle
\quad\qquad +\
\big[ \coend^{\vec{z_2}\in\freesmc\scat A}
  \freesmc\scat A(\vec x, (b) \concat \vec{z_2} )
  \times
  \freesmc\scat A(\vec{z_2}\concat (a), \vec y)
  \big]
\,+\,
  \big[
  \freesmc\scat A\big( (\,), (b) \big)
  \times
  \freesmc\scat A(\vec x\concat (a), \vec y)
  \big]
%\end{eqnarray}
\\[1.5mm]
\label{CombinatorialContentTwo}
&&\textstyle
\hspace{-22.5mm}
%\begin{eqnarray}
\quad\iso\
\big[
\scat A(a,b) \times \freesmc\scat A(\vec x,\vec y)
\big]
\,+ \,
\big[
\coend^{\vec z\in\freesmc\scat A}
  \freesmc\scat A(\vec x, (b)\concat \vec z)
  \times
  \freesmc\scat A(\vec z\concat(a),\vec y)
\big]
\\[2mm]
\nonumber
&&\textstyle
\hspace{-22.5mm}
\quad\iso\
I_{\scat A\times\freesmc\scat A}\big( (a,\vec x),(b,\vec y)\big)
\\[1mm]
\nonumber
&&\textstyle
\hspace{-22.5mm}
\quad\qquad
\,+ \,
\coend^{\vec z\in\freesmc\scat A,c,d\in\scat A}
  \freesmc\scat A(\vec x, (c)\concat \vec z)
  \times
  (\scat A\times\scat A)\,\big( (a,c),(d,b)\big)
  \times
  \freesmc\scat A(\vec z\concat(d),\vec y)
\\[2mm]
\nonumber
&&\textstyle
\hspace{-22.5mm}
\quad\iso\
  \big( I_{\scat A\times\freesmc\scat A} 
    +
    (I_\scat A\times\creation\eta_{\scat A})\icomp(\sigma_{\scat A,\scat
    A}\times I_{\freesmc\scat
    A})\icomp(I_{\scat A}\times\annihilation\epsilon_{\scat A}) \big)\,\big(
    (a,\vec x),(b,\vec y)\big)
\end{eqnarray}
This formal derivation can be pictorially represented as follows:\\[4mm]
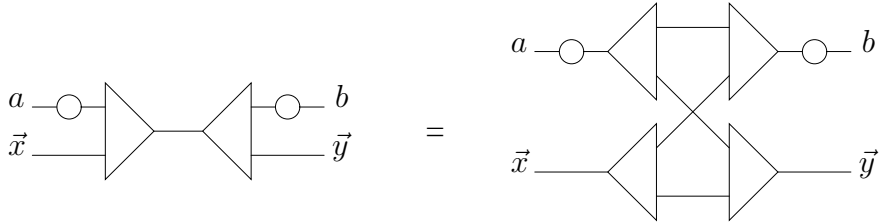
\begin{tabular}{l}
\begin{minipage}{5cm}
\input{diag1.eepic}
\end{minipage}
\quad=\qquad
\begin{minipage}{2.5cm}
\input{diag2.eepic}
\end{minipage}
\end{tabular}
\\
\begin{tabular}{l}
\hspace{-25mm}
=\quad
\begin{minipage}{5cm}
\input{diag3.eepic}
\end{minipage}
\\[15mm]
\hspace{-25mm}
=\quad
\begin{minipage}{5cm}
\input{diag4.eepic}
\end{minipage}
\end{tabular}
\\
\begin{tabular}{l}
=\quad
\begin{minipage}{5cm}
\input{diag5.eepic}
\end{minipage}
\end{tabular}

\subsection{Coherent states}
\label{CombinatorialCoherentStates}

In this section, I will indistinguishably regard profunctors
$\Tercat\profrightarrow\scat A$ as presheaves in $\Set^{\scat A}$, and
vice~versa.  Thus, according to
Definition~\ref{ExtensionsDef}~(\ref{GlobalElementExtension}), every
$V\in\Set^{\scat A}$ has a coherent state extension $\ext V\in
\Set^{\freesmc\scat A}$.  A calculation shows this to be given as
$$\textstyle
\ext V 
\,\iso \,
\coend^{\vec a \in\freesmc\scat A}
  \big( \prod_{a_i\in\vec a}\, V_{a_i} \big) 
  \, \cdot 
  \ket{\vec a}
%\enspace.
$$

The combinatorial version of the coherent state property of
Definition~\ref{CoherentStateDefinition}~(\ref{CoherentStateDefinitionOne})
enjoyed by $\ext V$ according to Theorem~\ref{CoherentStateTheorem} yields
the isomorphism
$$
(\Fun\ \annihilation\epsilon_{\scat A})(\ext V)_{(a,\vec x)}
\,\iso\,
V_a\times \ext V_{\vec x} 
$$
from which we obtain the functorial form
$$\textstyle
(\Fun\ \annihilation\epsilon_{\scat A})(\ext V)
\,\iso\,
\coend^{a\in\scat A,\vec x\in\freesmc\scat A}\,
  ( V_a\times \prod_{x_i\in\vec x} P_{x_i} )\, \cdot \ket{(a,\vec x)}
%\enspace.
$$

I now proceed to introduce a notion of \emph{exponential} (as parameterised by
algebras) and show how, when applied to the creation operator (with respect to
the free algebra), generalises the coherent state extension.  The definition
of exponential is based on that given in~\cite[Section~4]{Vicary}.

I have remarked in Section~2.2.5 %\ref{ProfDualityAndMonad} 
that $(\freesmc,\eta,\mu)$ is a pseudo monad on the bicategory of profunctors.
Pseudo algebras for it consist of profunctors $M:\freesmc\scat
A\profrightarrow\scat A$ equipped with natural isomorphisms 
$$
\xymatrix{
\ar@{}[dr]|-{\raisebox{4mm}{\hspace{5mm}\scriptsize$\iso$}}
\ar[rd]|-\shortmid_-{I_{\scat A}}
\scat A \ar[r]|-{\shortmid}^-{\eta_{\scat A}} & \freesmc\scat A
\ar[d]|-{-}^-{M} \\
& \scat A
}
\qquad\qquad
\xymatrix{
\ar@{}[dr]|-\iso
\ar[d]|-{-}_-{\mu_{\scat A}}
\freesmc\freesmc\scat A \ar[r]|-\shortmid^-{\freesmc M} &  \freesmc\scat A
\ar[d]^-M|-{-}
\\
\freesmc\scat A \ar[r]_M|-\shortmid & \scat A
}
$$
subject to coherence conditions, see~\eg~\cite{BKP}.  These pseudo
algebras provide the right notion of \emph{unbiased commutative
promonoidal category}, generalising the notion of \emph{symmetric
promonoidal category}~\cite{DayProMon} (\viz~commutative pseudo monoids in
the bicategory of profunctors) to biequivalent structures specified by
$n$-ary operations $M^{(n)}:\scat A^n\quotient{\symmgroup_n}
\profrightarrow\scat A$ for all $n\in\Nat$ that are commutative and
associative with unit $M^{(0)}$ up to coherent isomorphism.  The most
common examples of pseudo $\freesmc$-algebras arise from small symmetric
monoidal categories, say $(\scat M,\aone,\odot)$, by letting $\scat
M^\star\!:\freesmc\scat M\profrightarrow\scat M$ be given by $\scat
M^\star\big((x_1,\ldots,x_n),x\big)= \scat M(x_1\odot\cdots\odot x_n,x)$,
so that $\scat M^\star( (\,),x)=\scat M(\aone,x)$.  In particular, the
free pseudo algebra $\mu_{\scat A}:\freesmc\scat A\profrightarrow\scat A$
on $\scat A$ is obtained by this construction on the free symmetric
monoidal category $(\freesmc\scat A,(\,),\cdot)$ on $\scat A$. 

\begin{definition}
Let $M:\freesmc\scat A\profrightarrow\scat A$ be a pseudo $\freesmc$-algebra.
For $T:\freesmc\scat X\profrightarrow\scat A$, define 
$\exp_M(T)= M\icomp\Kext T: \freesmc\scat X\profrightarrow\scat A$.
\end{definition}
In particular, for $V\in\Set^{\scat A}$, we have that $$\textstyle
\exp_M(V) 
= \coend^{a\in\scat A} 
    \,\big[ \coend^{\vec x\in\freesmc\scat A} 
            \big(\prod_{x_i\in\vec x} V_{x_i}\big)\times 
      M(\vec x,a) 
    \big]\,\, 
    \cdot 
    \ket a
%\enspace.
$$

\begin{proposition}
For a pseudo $\freesmc$-algebra $M:\freesmc\scat A\profrightarrow\scat A$,
$$
\exp_M(0_{\mbox{\scriptsize$\Tercat$},\scat A}) \iso M^{(0)}
$$
and
$$ 
\exp_M(S+T)
=
(\xymatrix{
\freesmc\scat X \ar[r]|-\shortmid^-{\bangDiag_{\freesmc\scat X}}
&
\freesmc\scat X \times \freesmc\scat X 
\ar[rrr]|-\shortmid^-{\exp_M(S)\times\exp_M(T)}
&&& 
\freesmc\scat A \times \freesmc\scat A 
\ar[r]|-\shortmid^-{\bangcoDiag_{\freesmc\scat A}}
&
\freesmc\scat A 
})
$$ 
for all $S,T:\freesmc\scat X\profrightarrow\scat A$.
\end{proposition}

Note that the notion of exponential with respect to free algebras is a form of
\emph{comonadic/monadic convolution}, as for $T:\freesmc\scat
X\profrightarrow\freesmc\scat A$, the definition of $\exp_{\mu_{\scat A}}(T)$
amounts to the composite
\begin{equation}\label{LiftedMonadicConvolution}
\hfill
\xymatrix{
\freesmc\scat X \ar[r]^-{\delta_{\scat X}}|-\shortmid
& 
\freesmc\freesmc\scat X \ar[r]^-{\freesmc\,T}|-\shortmid
&
\freesmc\freesmc\scat A \ar[r]^-{\mu_{\scat A}}|-\shortmid
&
\freesmc\scat A 
}
%\enspace.
\hfill\hfill
\end{equation}

\begin{theorem}
For $V\in\Set^{\scat A}$,
$$
\exp_{\mu_{\scat A}}(\creation\eta_{\scat A}^V)
  \icomp \banginitialmap_{\scat A}
\iso
\ext V
%\enspace.
$$
\end{theorem}
\begin{proof}
A simple algebraic proof follows:
$$\begin{array}{l}
\exp_{\mu_{\scat A}}(\creation\eta_{\scat A}^V)
  \icomp\banginitialmap_{\scat A} 
\ = \
\mu_{\scat A}
  \icomp\Kext{(\creation\eta_{\scat A}^V)}
  \icomp\banginitialmap_{\scat A}
\ \iso \
\mu_{\scat A}\icomp\ext{\eta_A\icomp V}
\quad\mbox{, by~(\ref{CoherentStateCreationEquation})}
%Proposition~\ref{CoherentStateCreation}}
\\[2mm]
\qquad \iso\ 
\mu_{\scat A}\icomp\freesmc(\eta_{\scat A})\icomp\ext{V}
\ \iso\
\ext V
\quad\mbox{, by a monad law%\enspace.
}
\end{array}$$\\[-10.5mm]
\end{proof}

I conclude the paper with a formal combinatorial proof of this result.
Observe first that for the composite~(\ref{LiftedMonadicConvolution}), we
have:
$$\begin{array}{l}
(\mu_{\scat A}\icomp\freesmc(T)\icomp\delta_{\scat X})\,(\vec x,\vec a)
\\[3mm]
\quad\iso\
\coend^{\xi\in\freesmc\freesmc\scat X,\alpha\in\freesmc\freesmc\scat A}
\coend^{\vec z\in\freesmc(\freesmc\scat X^\op\times\freesmc\scat A)}
\begin{array}[t]{l}
  \big(\prod_{z_i\in\vec z} Tz_i\big)
  \times
  \freesmc\freesmc\scat X(\xi,\freesmc\proj1\vec z)
  \times
  \freesmc\freesmc\scat A(\freesmc\proj2\vec z,\alpha)
  \\[1mm] \quad
  \times\
  \freesmc\scat X(\vec x,\flatten\xi)
  \times
  \freesmc\scat A(\flatten\alpha,\vec a)
\end{array}
\\[8mm]
\quad\iso\
\coend^{\vec z\in\freesmc(\freesmc\scat X^\op\times\freesmc\scat A)}
  \big(\prod_{z_i\in\vec z} Tz_i\big)
  \times
  \freesmc\scat X\big(\vec x,\flatten{[\freesmc\proj1\vec z]}\big)
  \times
  \freesmc\scat A\big(\flatten{[\freesmc\proj2\vec z]},\vec a\big)
\end{array}$$
and hence that
$$\begin{array}{l}
(\mu_{\scat A}\icomp\freesmc(T)\icomp\delta_{\scat X}\icomp\banginitialmap_{\scat X})\,(\vec a)
\\[3mm]
\quad\iso\
\coend^{\vec z\in\freesmc(\freesmc\scat X^\op\times\freesmc\scat A)}
  \big(\prod_{z_i\in\vec z} Tz_i\big)
  \times
  \freesmc\scat X\big(\, (\,)\, ,\flatten{[\freesmc\proj1\vec z]}\big)
  \times
  \freesmc\scat A\big(\flatten{[\freesmc\proj2\vec z]},\vec a\big)
\\[3mm]
\quad\iso\
\coend^{\vec z\in\freesmc\freesmc\scat A}
  \big(\prod_{z_i\in\vec z} T( (\,),z_i)\big)
  \times
  \freesmc\scat A\big(\flatten{\vec z\,},\vec a\big)
\end{array}$$
Then, according to~(\ref{CreationCoendFormula}), 
$$\begin{array}{l}
(\mu_{\scat A}\icomp\freesmc(\creation\eta^V_{\scat A})\icomp\delta_{\scat A}\icomp\banginitialmap_{\scat A})\,(\vec a)
\\[3mm]
\quad\iso\
\coend^{\vec z\in\freesmc\freesmc\scat A}
  \big(\prod_{z_i\in\vec z} 
    \coend^{x\in\scat A} V_x\times\freesmc\scat A\big( (x),z_i\big)
  \big)
  \times
  \freesmc\scat A\big(\flatten{\vec z\,},\vec a\big)
\\[3mm]
\quad\iso\
\coend^{\vec z\in\freesmc\freesmc\scat A}
\coend^{x_{z_i}\in\scat A\,(z_i\in\vec z)} 
  \big( \prod_{z_i\in\vec z} V_{x_{z_i}} \big)
  \times
  \big(\prod_{z_i\in\vec z} \freesmc\scat A\big( (x_{z_i}),z_i \big)
  \big)
  \times
  \freesmc\scat A\big(\flatten{\vec z\,},\vec a\big)
\\[3mm]
\quad\iso\
\coend^{\vec x\in\freesmc\scat A}
  \big( \prod_{x_i\in\vec x} V_{x_i} \big)
  \times
  \freesmc\scat A\big(\flatten{\lift{\vec x}},\vec a\big)
\\[3mm]
\quad\iso\
\prod_{x_i\in\vec a} V_{x_i} 
\end{array}$$
where, for $a_i\in\scat A$, 
$\lift{(a_1,\ldots,a_n)}
 = \big(\, (a_1),\ldots,(a_n)\big)\in\freesmc\freesmc\scat A$;
so that, for $\vec a\in\freesmc\scat A$, $\flatten{\lift{\vec a}}=\vec a$.

\hide{%%% BEGIN HIDE !!!
For the combinatorial content of the argument, note first that
$$\textstyle
\big(\Fun\,\eta_{\scat A}\big)(V)
\,=\,
\coend^{\vec x}
  \big[ \coend^{a\in\scat A} 
          V_a \times \freesmc\scat A\big( (a),\vec x\,\big)\,\big]\,
  \cdot
  \ket{\vec x}
\,\iso\,
\coend^{a\in\scat A} V_a \cdot \ket{(a)}
$$
and, letting
$\flatten{(\alpha_1,\ldots,\alpha_n)}
 = \alpha_1\cdot\ldots\cdot\alpha_n\in\freesmc\scat A$
for $\alpha_i\in\freesmc\scat A$ and 
$\lift{(a_1,\ldots,a_n)}
 = \big(\, (a_1),\ldots,(a_n)\big)\in\freesmc\freesmc\scat A$ 
for $a_i\in\scat A$, 
calculate as follows:
$$\begin{array}{l}
\exp_{\mu_\scat A}(\eta_{\scat A}\,V)
\\[1mm]
\qquad\iso\
\exp_{\mu_\scat A}\big(\coend^{a\in\scat A} V_a\, \cdot \ket{(a)}\big)
\\[3mm]
\qquad=\
\coend^{\vec x\in\freesmc\scat A} 
    \,\big[ \coend^{\alpha\in\freesmc\freesmc\scat A} 
            \big(\prod_{\alpha_i\in\alpha} 
    \big(\coend^{a\in\scat A} V_a\ \cdot \ket{(a)} \big)_{\alpha_i}\big)
    \times \freesmc\scat A(\flatten\alpha,\vec x)\,\big]\,\, 
    \cdot 
    \ket{\vec x}
\\[3mm]
\qquad\iso\
\coend^{\alpha\in\freesmc\freesmc\scat A} 
    \big[\prod_{\alpha_i\in\alpha} 
    \big(\coend^{a\in\scat A} V_a\ \cdot \ket{(a)} \big)_{\alpha_i}\big]\,
    \cdot 
    \ket{\flatten \alpha}
\\[3mm]
\qquad\iso\
\coend^{\alpha\in\freesmc\freesmc\scat A} 
    \big[ \prod_{\alpha_i\in\alpha} 
          \big(\coend^{a\in\scat A} V_a \times 
          \freesmc\scat A\big( (a) , {\alpha_i}\big)\, \big)
    \big]\,
    \cdot 
    \ket{\flatten \alpha}
\\[3mm]
\qquad\iso\
\coend^{\alpha\in\freesmc\freesmc\scat A} 
  \coend^{a_{\alpha_i}\in\scat A\, (\alpha_i\in \alpha)} 
  \big[\,
  \big( \prod_{\alpha_i\in\alpha} V_{a_{\alpha_i}} \big)
  \times
  \prod_{\alpha_i\in\alpha} 
  \freesmc\scat A\big( (a_{\alpha_i}) , \alpha_i \big)
  \,\big]\
  \cdot 
  \ket{\flatten \alpha}
\\[3mm]
\qquad\iso\
\coend^{\vec a\in\freesmc\scat A} 
  \big( \prod_{a_i\in\vec a} V_{a_i} \big)\,
  \cdot 
  \ket{\flatten{\lift{\vec a} } }
\\[3mm]
\qquad=\
\coend^{\vec a\in\freesmc\scat A} 
  \big( \prod_{a_i\in\vec a} V_{a_i} 
  \,\big)\
  \cdot 
  \ket{\vec a}
\end{array}$$
}%%% END HIDE !!!

% Cahiers wants the author's address at the end of the paper:

\vspace{5mm}
\noindent
Marcelo Fiore\\
Computer Laboratory\\
University of Cambridge\\
15 JJ Thomson Avenue\\
Cambridge CB3 0FD\\
UK\\
Marcelo.Fiore@cl.cam.ac.uk

\end{document}

%% file: bialg1.eepic
\setlength{\unitlength}{0.00083333in}
\begingroup\makeatletter\ifx\SetFigFont\undefined%
\gdef\SetFigFont#1#2#3#4#5{%
  \reset@font\fontsize{#1}{#2pt}%
  \fontfamily{#3}\fontseries{#4}\fontshape{#5}%
  \selectfont}%
\fi\endgroup%
{\renewcommand{\dashlinestretch}{30}
\begin{picture}(2130,639)(0,-10)
\path(1515,612)(1515,12)(1215,312)(1515,612)
\path(615,612)(615,12)(915,312)(615,612)
\path(915,312)(1215,312)
\path(1815,462)(1965,462)
\path(1515,162)(1965,162)
\path(465,462)(615,462)
\path(165,462)(465,462)
\path(1515,462)(1665,462)(1815,462)
\path(165,162)(615,162)
\put(2115,162){\makebox(0,0)[rb]{\smash{{\SetFigFont{12}{14.4}{\familydefault}{\mddefault}{\updefault}$\vec
y$}}}}
\put(15,462){\makebox(0,0)[lb]{\smash{{\SetFigFont{12}{14.4}{\familydefault}{\mddefault}{\updefault}$\vec
u$}}}}
\put(15,162){\makebox(0,0)[lb]{\smash{{\SetFigFont{12}{14.4}{\familydefault}{\mddefault}{\updefault}$\vec
v$}}}}
\put(2115,462){\makebox(0,0)[rb]{\smash{{\SetFigFont{12}{14.4}{\familydefault}{\mddefault}{\updefault}$\vec
x$}}}}
\end{picture}
}

%% file: bialg2.eepic
\setlength{\unitlength}{0.00083333in}
\begingroup\makeatletter\ifx\SetFigFont\undefined%
\gdef\SetFigFont#1#2#3#4#5{%
  \reset@font\fontsize{#1}{#2pt}%
  \fontfamily{#3}\fontseries{#4}\fontshape{#5}%
  \selectfont}%
\fi\endgroup%
{\renewcommand{\dashlinestretch}{30}
\begin{picture}(2280,1389)(0,-10)
\path(915,612)(915,12)(615,312)(915,612)
\path(915,1362)(915,762)(615,1062)(915,1362)
\path(1365,612)(1365,12)(1665,312)(1365,612)
\path(1365,1362)(1365,762)(1665,1062)(1365,1362)
\path(915,1212)(1365,1212)
\path(915,162)(1365,162)
\path(915,462)(1365,912)
\path(915,912)(1365,462)
\path(1665,1062)(1965,1062)
\path(1965,1062)(2115,1062)
\path(1665,312)(2115,312)
\path(465,1062)(615,1062)
\path(165,1062)(465,1062)
\path(165,312)(615,312)
\put(2265,312){\makebox(0,0)[rb]{\smash{{\SetFigFont{12}{14.4}{\familydefault}{\mddefault}{\updefault}$\vec
y$}}}}
\put(15,1062){\makebox(0,0)[lb]{\smash{{\SetFigFont{12}{14.4}{\familydefault}{\mddefault}{\updefault}$\vec
u$}}}}
\put(15,312){\makebox(0,0)[lb]{\smash{{\SetFigFont{12}{14.4}{\familydefault}{\mddefault}{\updefault}$\vec
v$}}}}
\put(2265,1062){\makebox(0,0)[rb]{\smash{{\SetFigFont{12}{14.4}{\familydefault}{\mddefault}{\updefault}$\vec
x$}}}}
\end{picture}
}

%% file: diag1.eepic
\setlength{\unitlength}{0.00083333in}
\begingroup\makeatletter\ifx\SetFigFont\undefined%
\gdef\SetFigFont#1#2#3#4#5{%
  \reset@font\fontsize{#1}{#2pt}%
  \fontfamily{#3}\fontseries{#4}\fontshape{#5}%
  \selectfont}%
\fi\endgroup%
{\renewcommand{\dashlinestretch}{30}
\begin{picture}(2130,639)(0,-10)
\put(1740,462){\ellipse{150}{150}}
\put(390,462){\ellipse{150}{150}}
\path(1515,612)(1515,12)(1215,312)(1515,612)
\path(615,612)(615,12)(915,312)(615,612)
\path(915,312)(1215,312)
\path(1515,462)(1665,462)
\path(1815,462)(1965,462)
\path(1515,162)(1965,162)
\path(465,462)(615,462)
\path(165,462)(315,462)
\path(165,162)(615,162)
\put(2115,462){\makebox(0,0)[rb]{\smash{{\SetFigFont{12}{14.4}{\familydefault}{\mddefault}{\updefault}$b$}}}}
\put(2115,162){\makebox(0,0)[rb]{\smash{{\SetFigFont{12}{14.4}{\familydefault}{\mddefault}{\updefault}$\vec y$}}}}
\put(15,462){\makebox(0,0)[lb]{\smash{{\SetFigFont{12}{14.4}{\familydefault}{\mddefault}{\updefault}$a$}}}}
\put(15,162){\makebox(0,0)[lb]{\smash{{\SetFigFont{12}{14.4}{\familydefault}{\mddefault}{\updefault}$\vec x$}}}}
\end{picture}
}

%% file: diag2.eepic
\setlength{\unitlength}{0.00083333in}
\begingroup\makeatletter\ifx\SetFigFont\undefined%
\gdef\SetFigFont#1#2#3#4#5{%
  \reset@font\fontsize{#1}{#2pt}%
  \fontfamily{#3}\fontseries{#4}\fontshape{#5}%
  \selectfont}%
\fi\endgroup%
{\renewcommand{\dashlinestretch}{30}
\begin{picture}(2280,1389)(0,-10)
\put(1890,1062){\ellipse{150}{150}}
\put(390,1062){\ellipse{150}{150}}
\path(915,612)(915,12)(615,312)(915,612)
\path(915,1362)(915,762)(615,1062)(915,1362)
\path(1365,612)(1365,12)(1665,312)(1365,612)
\path(1365,1362)(1365,762)(1665,1062)(1365,1362)
\path(915,1212)(1365,1212)
\path(915,162)(1365,162)
\path(915,462)(1365,912)
\path(915,912)(1365,462)
\path(1665,1062)(1815,1062)
\path(1965,1062)(2115,1062)
\path(1665,312)(2115,312)
\path(465,1062)(615,1062)
\path(165,1062)(315,1062)
\path(165,312)(615,312)
\put(2265,312){\makebox(0,0)[rb]{\smash{{\SetFigFont{12}{14.4}{\familydefault}{\mddefault}{\updefault}$\vec y$}}}}
\put(2265,1062){\makebox(0,0)[rb]{\smash{{\SetFigFont{12}{14.4}{\familydefault}{\mddefault}{\updefault}$b$}}}}
\put(15,1062){\makebox(0,0)[lb]{\smash{{\SetFigFont{12}{14.4}{\familydefault}{\mddefault}{\updefault}$a$}}}}
\put(15,312){\makebox(0,0)[lb]{\smash{{\SetFigFont{12}{14.4}{\familydefault}{\mddefault}{\updefault}$\vec x$}}}}
\end{picture}
}

%% file: diag3.eepic
\setlength{\unitlength}{0.00083333in}
\begingroup\makeatletter\ifx\SetFigFont\undefined%
\gdef\SetFigFont#1#2#3#4#5{%
  \reset@font\fontsize{#1}{#2pt}%
  \fontfamily{#3}\fontseries{#4}\fontshape{#5}%
  \selectfont}%
\fi\endgroup%
{\renewcommand{\dashlinestretch}{30}
\begin{picture}(7905,1311)(0,-10)
\put(3690,312){\makebox(0,0)[rb]{\smash{{\SetFigFont{12}{14.4}{\familydefault}{\mddefault}{\updefault}$\vec y$}}}}
\put(4515,462){\ellipse{150}{150}}
\path(4215,912)(4515,912)
\path(4815,1212)(4815,612)(4515,912)(4815,1212)
\path(5340,1062)(4815,1062)
\path(4815,762)(5115,762)
\path(5490,1062)(5715,1062)
\path(5115,912)(5115,312)(5415,612)(5115,912)
\path(5415,612)(5715,612)
\path(4590,462)(5115,462)
\path(4215,462)(4440,462)
\put(4065,912){\makebox(0,0)[lb]{\smash{{\SetFigFont{12}{14.4}{\familydefault}{\mddefault}{\updefault}$\vec x$}}}}
\put(5865,1062){\makebox(0,0)[rb]{\smash{{\SetFigFont{12}{14.4}{\familydefault}{\mddefault}{\updefault}$b$}}}}
\put(5865,612){\makebox(0,0)[rb]{\smash{{\SetFigFont{12}{14.4}{\familydefault}{\mddefault}{\updefault}$\vec y$}}}}
\put(4065,462){\makebox(0,0)[lb]{\smash{{\SetFigFont{12}{14.4}{\familydefault}{\mddefault}{\updefault}$a$}}}}
\put(6690,237){\ellipse{150}{150}}
\path(6915,987)(6915,387)(6615,687)(6915,987)
\path(6390,687)(6615,687)
\path(7140,837)(6915,837)
\path(7140,912)(7140,762)
\path(6915,537)(7215,537)
\path(7215,687)(7215,87)(7515,387)(7215,687)
\path(6765,237)(7215,237)
\path(6390,237)(6615,237)
\path(7515,387)(7740,387)
\put(6240,687){\makebox(0,0)[lb]{\smash{{\SetFigFont{12}{14.4}{\familydefault}{\mddefault}{\updefault}$\vec x$}}}}
\put(6240,237){\makebox(0,0)[lb]{\smash{{\SetFigFont{12}{14.4}{\familydefault}{\mddefault}{\updefault}$a$}}}}
\put(7890,387){\makebox(0,0)[rb]{\smash{{\SetFigFont{12}{14.4}{\familydefault}{\mddefault}{\updefault}$\vec y$}}}}
\put(3090,1137){\ellipse{150}{150}}
\put(7440,1137){\ellipse{150}{150}}
\put(1140,762){\ellipse{150}{150}}
\put(465,1137){\ellipse{134}{150}}
\put(2640,1137){\ellipse{150}{150}}
\path(3015,1137)(2715,1137)
\path(3165,1137)(3540,1137)
\path(7215,1212)(7215,1062)
\path(765,87)(915,87)
\path(690,462)(915,462)
\path(1065,762)(690,762)
\path(1215,762)(1440,762)
\path(557,1137)(690,1137)
\path(690,1212)(690,1062)
\path(165,1137)(390,1137)
\path(915,612)(915,12)(1215,312)(915,612)
\path(1215,312)(1440,312)
\path(765,162)(765,87)(765,12)
\path(2190,1137)(2565,1137)
\path(165,612)(390,612)
\path(690,912)(690,312)(390,612)(690,912)
\path(2190,612)(2415,612)
\path(2715,912)(2715,312)(2415,612)(2715,912)
\path(2865,762)(2715,762)
\path(2865,837)(2865,687)
\path(2715,537)(3015,537)
\path(3015,612)(3015,12)(3315,312)(3015,612)
\path(2865,162)(3015,162)
\path(2865,237)(2865,87)
\path(3315,312)(3540,312)
\path(7365,1137)(7215,1137)
\path(7515,1137)(7740,1137)
\put(1740,687){\makebox(0,0)[lb]{\smash{{\SetFigFont{12}{14.4}{\rmdefault}{\mddefault}{\updefault}+}}}}
\put(3690,1137){\makebox(0,0)[rb]{\smash{{\SetFigFont{12}{14.4}{\familydefault}{\mddefault}{\updefault}$b$}}}}
\put(3840,687){\makebox(0,0)[lb]{\smash{{\SetFigFont{12}{14.4}{\rmdefault}{\mddefault}{\updefault}+}}}}
\put(5940,687){\makebox(0,0)[lb]{\smash{{\SetFigFont{12}{14.4}{\rmdefault}{\mddefault}{\updefault}+}}}}
\put(7890,1137){\makebox(0,0)[rb]{\smash{{\SetFigFont{12}{14.4}{\familydefault}{\mddefault}{\updefault}$b$}}}}
\put(1590,762){\makebox(0,0)[rb]{\smash{{\SetFigFont{12}{14.4}{\familydefault}{\mddefault}{\updefault}$b$}}}}
\put(1590,312){\makebox(0,0)[rb]{\smash{{\SetFigFont{12}{14.4}{\familydefault}{\mddefault}{\updefault}$\vec y$}}}}
\put(2040,1137){\makebox(0,0)[lb]{\smash{{\SetFigFont{12}{14.4}{\familydefault}{\mddefault}{\updefault}$a$}}}}
\put(15,1137){\makebox(0,0)[lb]{\smash{{\SetFigFont{11}{13.2}{\familydefault}{\mddefault}{\updefault}$a$}}}}
\put(15,612){\makebox(0,0)[lb]{\smash{{\SetFigFont{12}{14.4}{\familydefault}{\mddefault}{\updefault}$\vec x$}}}}
\put(2040,612){\makebox(0,0)[lb]{\smash{{\SetFigFont{12}{14.4}{\familydefault}{\mddefault}{\updefault}$\vec x$}}}}
\put(5415,1062){\ellipse{150}{150}}
\end{picture}
}

%% file: diag4.eepic
\setlength{\unitlength}{0.00083333in}
\begingroup\makeatletter\ifx\SetFigFont\undefined%
\gdef\SetFigFont#1#2#3#4#5{%
  \reset@font\fontsize{#1}{#2pt}%
  \fontfamily{#3}\fontseries{#4}\fontshape{#5}%
  \selectfont}%
\fi\endgroup%
{\renewcommand{\dashlinestretch}{30}
\begin{picture}(7380,1092)(0,-10)
\put(3615,312){\makebox(0,0)[lb]{\smash{{\SetFigFont{12}{14.4}{\familydefault}{\mddefault}{\updefault}$a$}}}}
\path(6465,912)(6915,912)
\path(7065,912)(7215,912)
\path(6465,987)(6465,837)
\put(7365,912){\makebox(0,0)[rb]{\smash{{\SetFigFont{12}{14.4}{\familydefault}{\mddefault}{\updefault}$b$}}}}
\put(1140,462){\ellipse{150}{150}}
\put(390,837){\ellipse{150}{150}}
\put(2340,837){\ellipse{150}{150}}
\put(2790,837){\ellipse{150}{150}}
\put(6240,237){\ellipse{150}{150}}
\put(4965,912){\ellipse{150}{150}}
\put(4065,312){\ellipse{150}{150}}
\path(2415,837)(2565,837)
\path(915,612)(915,12)(615,312)(915,612)
\path(165,312)(615,312)
\path(915,162)(1365,162)
\path(1065,462)(915,462)
\path(1215,462)(1365,462)
\path(465,837)(915,837)
\path(165,837)(315,837)
\path(915,912)(915,762)
\path(2115,837)(2265,837)
\path(2115,237)(3015,237)
\path(2565,837)(2715,837)
\path(2865,837)(3015,837)
\path(6465,687)(6465,87)(6765,387)(6465,687)
\path(6765,387)(7215,387)
\path(6015,537)(6465,537)
\path(6315,237)(6465,237)
\path(6015,237)(6165,237)
\path(3765,762)(4065,762)
\path(4365,1062)(4365,462)(4065,762)(4365,1062)
\path(4890,912)(4365,912)
\path(4365,612)(4665,612)
\path(5040,912)(5265,912)
\path(4665,762)(4665,162)(4965,462)(4665,762)
\path(4965,462)(5265,462)
\path(4140,312)(4665,312)
\path(3765,312)(3990,312)
\put(1665,537){\makebox(0,0)[lb]{\smash{{\SetFigFont{12}{14.4}{\rmdefault}{\mddefault}{\updefault}+}}}}
\put(3315,537){\makebox(0,0)[lb]{\smash{{\SetFigFont{12}{14.4}{\rmdefault}{\mddefault}{\updefault}+}}}}
\put(5565,537){\makebox(0,0)[lb]{\smash{{\SetFigFont{12}{14.4}{\rmdefault}{\mddefault}{\updefault}+}}}}
\put(15,312){\makebox(0,0)[lb]{\smash{{\SetFigFont{12}{14.4}{\familydefault}{\mddefault}{\updefault}$\vec x$}}}}
\put(1515,162){\makebox(0,0)[rb]{\smash{{\SetFigFont{12}{14.4}{\familydefault}{\mddefault}{\updefault}$\vec y$}}}}
\put(1515,462){\makebox(0,0)[rb]{\smash{{\SetFigFont{12}{14.4}{\familydefault}{\mddefault}{\updefault}$b$}}}}
\put(15,837){\makebox(0,0)[lb]{\smash{{\SetFigFont{12}{14.4}{\familydefault}{\mddefault}{\updefault}$a$}}}}
\put(1965,837){\makebox(0,0)[lb]{\smash{{\SetFigFont{12}{14.4}{\familydefault}{\mddefault}{\updefault}$a$}}}}
\put(1965,237){\makebox(0,0)[lb]{\smash{{\SetFigFont{12}{14.4}{\familydefault}{\mddefault}{\updefault}$\vec x$}}}}
\put(3165,237){\makebox(0,0)[rb]{\smash{{\SetFigFont{12}{14.4}{\familydefault}{\mddefault}{\updefault}$\vec y$}}}}
\put(3165,837){\makebox(0,0)[rb]{\smash{{\SetFigFont{12}{14.4}{\familydefault}{\mddefault}{\updefault}$b$}}}}
\put(7365,387){\makebox(0,0)[rb]{\smash{{\SetFigFont{12}{14.4}{\familydefault}{\mddefault}{\updefault}$\vec y$}}}}
\put(5865,537){\makebox(0,0)[lb]{\smash{{\SetFigFont{12}{14.4}{\familydefault}{\mddefault}{\updefault}$\vec x$}}}}
\put(5865,237){\makebox(0,0)[lb]{\smash{{\SetFigFont{12}{14.4}{\familydefault}{\mddefault}{\updefault}$a$}}}}
\put(3615,762){\makebox(0,0)[lb]{\smash{{\SetFigFont{12}{14.4}{\familydefault}{\mddefault}{\updefault}$\vec x$}}}}
\put(5415,912){\makebox(0,0)[rb]{\smash{{\SetFigFont{12}{14.4}{\familydefault}{\mddefault}{\updefault}b}}}}
\put(5415,462){\makebox(0,0)[rb]{\smash{{\SetFigFont{12}{14.4}{\familydefault}{\mddefault}{\updefault}$\vec y$}}}}
\put(6990,912){\ellipse{150}{150}}
\end{picture}
}

%% file: diag5.eepic
\setlength{\unitlength}{0.00083333in}
\begingroup\makeatletter\ifx\SetFigFont\undefined%
\gdef\SetFigFont#1#2#3#4#5{%
  \reset@font\fontsize{#1}{#2pt}%
  \fontfamily{#3}\fontseries{#4}\fontshape{#5}%
  \selectfont}%
\fi\endgroup%
{\renewcommand{\dashlinestretch}{30}
\begin{picture}(3480,939)(0,-10)
\put(1665,612){\makebox(0,0)[lb]{\smash{{\SetFigFont{12}{14.4}{\familydefault}{\mddefault}{\updefault}$\vec x$}}}}
\put(2115,162){\ellipse{150}{150}}
\path(1815,612)(2115,612)
\path(2415,912)(2415,312)(2115,612)(2415,912)
\path(2940,762)(2415,762)
\path(2415,462)(2715,462)
\path(3090,762)(3315,762)
\path(2715,612)(2715,12)(3015,312)(2715,612)
\path(3015,312)(3315,312)
\path(2190,162)(2715,162)
\path(1815,162)(2040,162)
\path(165,762)(1065,762)
\path(165,162)(1065,162)
\put(3465,762){\makebox(0,0)[rb]{\smash{{\SetFigFont{12}{14.4}{\familydefault}{\mddefault}{\updefault}$b$}}}}
\put(3465,312){\makebox(0,0)[rb]{\smash{{\SetFigFont{12}{14.4}{\familydefault}{\mddefault}{\updefault}$y$}}}}
\put(15,762){\makebox(0,0)[lb]{\smash{{\SetFigFont{12}{14.4}{\familydefault}{\mddefault}{\updefault}$a$}}}}
\put(1215,762){\makebox(0,0)[rb]{\smash{{\SetFigFont{12}{14.4}{\familydefault}{\mddefault}{\updefault}$b$}}}}
\put(15,162){\makebox(0,0)[lb]{\smash{{\SetFigFont{12}{14.4}{\familydefault}{\mddefault}{\updefault}$\vec x$}}}}
\put(1215,162){\makebox(0,0)[rb]{\smash{{\SetFigFont{12}{14.4}{\familydefault}{\mddefault}{\updefault}$\vec y$}}}}
\put(1365,462){\makebox(0,0)[lb]{\smash{{\SetFigFont{12}{14.4}{\rmdefault}{\mddefault}{\updefault}+}}}}
\put(1665,162){\makebox(0,0)[lb]{\smash{{\SetFigFont{12}{14.4}{\familydefault}{\mddefault}{\updefault}$a$}}}}
\put(3015,762){\ellipse{150}{150}}
\end{picture}
}